\documentclass[10pt,leqno]{amsart}
   \usepackage[centertags]{amsmath}
   \usepackage{amsfonts}
   \usepackage{amsmath}
   \usepackage{amssymb}
   \usepackage{amsthm}
   \usepackage{newlfont}
   \usepackage[latin1]{inputenc}
\usepackage{multirow, bigdelim}

\usepackage{graphicx}

\usepackage{bm}

\allowdisplaybreaks[3]

\theoremstyle{plain}
\newtheorem{theorem}{Theorem}[section]
\newtheorem{lemma}[theorem]{Lemma}

\newtheorem{corollary}[theorem]{Corollary}

\theoremstyle{definition}

\newtheorem{example}[theorem]{Example}
\theoremstyle{remark}
\newtheorem{remark}[theorem]{Remark}
\newtheorem*{remark*}{Remark}

\numberwithin{equation}{section}

\newcommand\D{{\mathcal D}}
\newcommand\A{{\mathcal A}}

\newcommand\RR{{\mathbb R}}
\newcommand\ZZ{{\mathbb Z}}
\newcommand\NN{{\mathbb N}}
\newcommand\PP{{\mathbb P}}

\newcommand\R{{\mathfrak R}}
\newcommand\Ss{{\mathfrak S}}

\newcommand\no{{\rho}}
\newcommand\Rc{{\mathcal R}}
\newcommand\alp{\tilde{\mathfrak D}}
\newcommand\al{\mathfrak D}

\newcommand\Sh{\mbox{\Large $\mathfrak {s}$}}

\catcode`,\active

\catcode`\,12

\title[]{Bispectral Jacobi type polynomials}

\author{Antonio J. Dur\'an and Manuel D. de la Iglesia}

\address{Antonio J. Dur\'an\\
Departamento de An\'{a}lisis Matem\'{a}tico \\
Universidad de Sevilla \\
Apdo (P. O. BOX) 1160\\
41080 Sevilla. Spain.}
\email{duran@us.es}
\address{Manuel D. de la Iglesia\\
Instituto de Matem\'aticas \\
Universidad Nacional Aut\'onoma de M\'exico \\
Circuito Exterior, C.U.\\
04510, Mexico D.F. Mexico.}
\email{mdi29@im.unam.mx}
\thanks{Partially supported by MTM2015-65888-C4-1-P (Ministerio de Econom\'ia y Competitividad), FQM-262 (Junta de Andaluc\'ia), Feder Funds (European Union), PAPIIT-DGAPA-UNAM grant IN104219 (M\'exico) and CONACYT grant A1-S-16202 (M\'exico).}

\date{\today}



\begin{document}
  \maketitle

   \begin{abstract}
We study the bispectrality of Jacobi type polynomials, which are eigenfunctions of higher-order differential operators and can be defined by taking suitable linear combinations of a fixed number of consecutive Jacobi polynomials. Jacobi type polynomials include, as particular cases, the  Krall-Jacobi polynomials. As the main results we prove that the  Jacobi type polynomials always satisfy higher-order recurrence relations (i.e., they are bispectral). We also prove that the Krall-Jacobi families are the only Jacobi type polynomials which are orthogonal with respect to a measure on the real line.
\end{abstract}

\section{Introduction and results}
Bispectrality in its continuous-continuous version is a subject that was started by H. Duistermaat and F.A. Gr\"unbaum in the 1980s \cite{DuiGr}. In the context of orthogonal polynomials, we say that a sequence of polynomials $(q_n(x))_n$ is bispectral if there exist a difference operator, acting on the discrete variable $n$, of the form
\begin{equation}\label{doho2g}
D_n=\sum_{i=s}^r\gamma_{n,i}\Sh _i, \quad s\le r, \quad s,r\in \ZZ,
\end{equation}
where $\Sh_l$ stands for the shift operator $\Sh_l(f(n))=f(n+l)$ and $\gamma_{n,i}$, $i=s,\ldots, r$, are sequences of numbers with $\gamma_{n,s},\gamma_{n,r}\not =0 $, $n\ge 0$, and a differential operator acting on the continuous variable $x$, with respect to which the polynomials $(q_n(x))_n$ are eigenfunctions (other type of operators acting on the continuous variable $x$ can be considered, but in this paper we restrict ourselves to differential operators).

It is easy to see that if $D_n(q_n)=Q(x)q_n$ then $Q$ is a polynomial of degree $r$, and hence each operator $D_n$ of the form (\ref{doho2g}) produces a higher-order recurrence relation for the polynomials $(q_n)_n$, i.e.
\begin{equation}\label{qrr}
Q(x)q_n(x)=\sum_{i=s}^r\gamma_{n,i}q_{n+i}(x), \quad s\le r.
\end{equation}
For $r=-s=1$, the recurrence relation (\ref{qrr}) reduces to the usual three-term recurrence relation for orthogonal polynomials with respect to a measure supported on the real line
\begin{equation}\label{ttrr}
xq_n(x)=a_nq_{n+1}(x)+b_nq_n(x)+c_nq_{n-1}(x), \quad n\ge 0,\quad q_{-1}=0.
\end{equation}
Hence the classical families of orthogonal polynomials, Hermite, Laguerre and Jacobi (and Bessel, if non-positive measures are considered), are examples of bispectral polynomials.

Krall polynomials are other well-known examples of bispectral polynomials. Krall polynomials are  eigenfunctions of higher-order differential operators. They are called Krall polynomials because they were introduced by H.L. Krall in 1940 \cite{Kr2}: Krall proved that the differential operators must have even order and classified the case of order four. Since the 1980's, Krall polynomials associated with differential operators of any even order have been constructed and intensively studied (\cite{koekoe,koe,koekoe2,L1,L2,GrH1,GrHH,GrY,Plamen1,Plamen2,Zh}; the list is not exhaustive). There are two known classes of Krall polynomials: the Krall-Laguerre and the Krall-Jacobi families. Krall-Laguerre polynomials are orthogonal with respect to measures of the form
\begin{equation*}\label{Kcwm}
x^{\alpha-m}e^{-x}+\sum_{h=0}^{m-1}b_h\delta_0^{(h)},\quad x\ge 0,
\end{equation*}
where $\alpha $ and $m$ are positive integers with $\alpha\geq m$ and $b_h$, $h=0,\ldots, m-1$, are certain real numbers with $b_{m-1}\not=0$. Krall-Jacobi polynomials are orthogonal with respect to any of the following measures
\begin{align}\label{kjm1}
&(1-x)^{\alpha-m_2}(1+x)^\beta+\sum_{h=0}^{m_2-1}c_h\delta_1^{(h)}, \quad \alpha\in \NN, \alpha \ge m_2,\\\label{kjm2}
&(1-x)^{\alpha}(1+x)^{\beta-m_1}+\sum_{h=0}^{m_1-1}c_h\delta_{-1}^{(h)}, \quad \beta\in \NN, \beta \ge m_1,\\\label{kjm3}
&(1-x)^{\alpha-m_2}(1+x)^{\beta-m_1}+\sum_{h=0}^{m_2-1}c_h\delta_1^{(h)}+\sum_{h=0}^{m_1-1}d_h\delta_{-1}^{(h)}, \quad \alpha,\beta\in \NN, \alpha \ge m_2, \beta \ge m_1.
\end{align}
Both, the Krall-Laguerre and Krall-Jacobi polynomials, are also eigenfunctions of a higher-order differential operator.

Other examples of bispectral polynomials are the Krall-Sobolev polynomials (see \cite{KKB,Ba,ddI1,ddI3}), the exceptional polynomials (see \cite{GUKM1,duch,dume,durr,duhj,GFGM}, and references therein) or the Gr\"unbaum and Haine extension of Krall polynomials (\cite{GrH3}; see also \cite{Plamen1,du1}). In these cases, the associated operators (in the discrete and continuous variable) have order greater than $2$.

In \cite{ddIlb}, we have studied Laguerre type polynomials. They are defined by taking suitable linear combinations of a fixed number of consecutive Laguerre  polynomials. These Laguerre  type polynomials are eigenfunctions of higher-order differential operators and include, as particular cases, the Krall-Laguerre polynomials. Among other things, we have proved in \cite{ddIlb} that Laguerre type polynomials are also bispectral and that the Krall-Laguerre families are the only Laguerre type polynomials which are orthogonal with respect to a measure on the real line.

The purpose of this paper is to study Jacobi type polynomials.
For $\alpha,\beta,\alpha+\beta \not =-1,-2,\ldots$ we use the following renormalization of the Jacobi polynomials:
$$
J_n^{\alpha,\beta}(x)=\frac{(-1)^n(\alpha+\beta+1)_n}{2^n(\beta+1)_n}
\sum_{j=0}^n\binom{n+\alpha}{j}\binom{n+\beta}{n-j}(x-1)^{n-j}(x+1)^j.
$$
We denote by $\mu_{\alpha,\beta}(x)$ the orthogonalizing weight for the Jacobi polynomials normalized so that
$
\int \mu_{\alpha,\beta}(x)dx=2^{\alpha+\beta+1}\frac{\Gamma(\alpha+1)\Gamma(\beta+1)}{\Gamma(\alpha+\beta +2)}.
$
Only when $\alpha,\beta >-1$, $\mu_{\alpha,\beta }(x)$, $-1<x<1$, is positive, and then
\begin{equation}\label{pJac}
\mu_{\alpha,\beta}(x) =(1-x)^\alpha (1+x)^\beta,\quad -1<x<1.
\end{equation}

From the Jacobi polynomials $(J_n^{\alpha,\beta})_n$, we can generate sequences of polynomials $(q_n(x))_n$ which are eigenfunctions of a higher-order differential operator (acting on the continuous variable $x$) in the following way.
 Consider two finite sets $G=\{g_1,\ldots,g_{m_1}\}$ and $H=\{h_1,\ldots,h_{m_2}\}$ of positive integers (written in increasing size) and polynomials $\mathcal{R}_g,g\in G$, with $\deg\mathcal{R}_g=g$ and $\mathcal{S}_h,h\in H$, with $\deg\mathcal{S}_h=h$. The positive integers $m_1$ and $m_2$ are the number of elements of $G$ and $H$, respectively, and let us call $m=m_1+m_2$. Since the leading coefficients of the polynomials $\Rc_g$ and $\mathcal{S}_h$ just produce a renormalization of the polynomials $(q_n)_n$ (see (\ref{iquss}) below), we assume along the rest of this paper that $\Rc_g(x)$ and $\mathcal{S}_h(x)$ are monic polynomials.

We also denote by $\mathcal{Z}_i, i=1,\ldots,m$, the set of polynomials defined by
\begin{equation}\label{defpol}
\mathcal{Z}_i(x)=\begin{cases}\mathcal{R}_{g_i}(x),&\mbox{for $i=1,\ldots, m_1$,}\\
\mathcal{S}_{h_{i-m_1}}(x),&\mbox{for $i=m_1+1,\ldots ,m$.}
\end{cases}
\end{equation}
For $\alpha-m_2\neq -1, -2, \ldots,$ and $\beta-m_1\neq -1, -2, \ldots,$ we write
\begin{equation}\label{defno}
\no ^i_{x,j}=\begin{cases} (-1)^{m-j}\Gamma^{\alpha-j,\beta-1}_{\alpha-m,\beta-j}(x),&\mbox{for $i=1,\ldots, m_1$,}\\
1,&\mbox{for $i=m_1+1,\ldots ,m$,}
\end{cases}
\end{equation}
where for $a,b,c,d,x\in \RR$, we define
\begin{equation}\label{gammas}
\Gamma^{a,b}_{c,d}(x)=\frac{\Gamma(x+a+1)\Gamma(x+b+1)}{\Gamma(x+c+1)\Gamma(x+d+1)}.
\end{equation}
We will always use $\no ^i_{x,j}$, with $j=0,\ldots ,m$. Hence when $\alpha$ or $\beta$ are nonnegative integers $\no ^i_{n,j}$, $n\in \NN$, can be also defined
from (\ref{gammas}) by using the standard properties of the Gamma function and taking $1/\Gamma(-n)=0$, $n\in \NN$.

We associate to $G$ and $H$ the following $m\times m$ quasi-Casoratian determinant
\begin{equation}\label{casdet}
\Lambda_{G,H}(n)=\frac{\begin{vmatrix}
\no^1_{n,1} \mathcal{Z}_1(\theta_{n-1}) &\no^1_{n,2} \mathcal{Z}_1(\theta_{n-2}) & \ldots &  \no^1_{n,m}\mathcal{Z}_1(\theta_{n-m}) \\
               \vdots & \vdots & \ddots & \vdots \\
\no^m_{n,1} \mathcal{Z}_m(\theta_{n-1}) & \no^m_{n,2} \mathcal{Z}_m(\theta_{n-2}) & \ldots & \no^m_{n,m}\mathcal{Z}_m(\theta_{n-m})
             \end{vmatrix}}{\mathfrak p(n)\mathfrak q(n)},
\end{equation}
where $\mathfrak p$ and $\mathfrak q$ are the following polynomials
\begin{align}\label{def1pi}
\mathfrak p(x)&=\prod_{i=1}^{m_1-1}(-1)^{m_1-i}(x+\alpha-m+1)_{m_1-i}(x+\beta-m_1+i)_{m_1-i},\\
\label{def1qi}\mathfrak q(x)&=(-1)^{\binom{m}{2}}\prod_{h=1}^{m-1}\left(\prod_{i=1}^{h}(2(x-m)+\alpha+\beta+i+h)\right),
\end{align}
where $(a)_0=1, (a)_n=a(a+1)\cdots(a+n-1)$ denotes as usual the Pochhammer symbol and $\theta_n=n(n+\alpha+\beta+1)$ is the eigenvalue associated with the second-order differential operator for the Jacobi polynomials.
Along this paper we will assume that
\begin{equation}\label{assum}
\Lambda_{G,H}(n)\neq0,\quad n=0,1,2,\ldots
\end{equation}
We then define the sequence of polynomials $(q_n)_n$ by
\begin{equation}\label{iquss}
q_n(x)=\frac{\begin{vmatrix}
               J_{n}^{\alpha,\beta}(x) & -J_{n-1}^{\alpha,\beta}(x) & \ldots & (-1)^mJ_{n-m}^{\alpha,\beta}(x) \\
\no^1_{n,0}\mathcal{Z}_1(\theta_{n}) &\no^1_{n,1} \mathcal{Z}_1(\theta_{n-1}) & \ldots &  \no^1_{n,m}\mathcal{Z}_1(\theta_{n-m}) \\
               \vdots & \vdots & \ddots & \vdots \\
\no^m_{n,0}\mathcal{Z}_m(\theta_{n}) & \no^m_{n,1} \mathcal{Z}_m(\theta_{n-1}) & \ldots & \no^m_{n,m}\mathcal{Z}_m(\theta_{n-m})
             \end{vmatrix}}{\mathfrak p(n)\mathfrak q(n)}
\end{equation}
(see the Remark \ref{qsc} for a discussion of how to define $\Lambda _{G,H}$ and $q_n$ when $\alpha$ and $\beta$ are nonnegative integers).
The assumption (\ref{assum}) says that the determinant on the right-hand side of (\ref{iquss}) defines a polynomial of degree $n$, $n\ge 0$. Expanding the determinant by its first row, we see that each $q_n$, $n\ge m$, is a linear combination of $m$ consecutive Jacobi polynomials.

Using the $\D$-operator method, it is proved in \cite{ddI3} (see Theorem 3.1 and the beginning of Section 4 of that paper) that the polynomials $(q_n)_n$ are eigenfunctions of a higher-order differential operator (acting on the continuous variable $x$) of the form
$D_x=\sum_{l=0}^rh_l(x)\left(\frac{d}{dx}\right)^l$,
where $h_l(x)$ are polynomials and $r$ is a positive even integer greater than $2$. This differential operator can, in fact, be explicitly constructed. For a different approach of the polynomials (\ref{iquss}) using discrete Darboux transformations see \cite{GrY,Plamen1}.

The most interesting case corresponds with the different families of Krall-Jacobi polynomials, orthogonal with respect to any of the measures (\ref{kjm1}), (\ref{kjm2}) or (\ref{kjm3}). Indeed, let $u_j^{\lambda}(x)$ be the following polynomials
\begin{equation}\label{basu}
u_j^{\lambda}(x)=(x+\alpha-\lambda+1)_j(x+\beta+\lambda-j+1)_j=
\Gamma_{\alpha-\lambda,\beta+\lambda-j}^{\alpha-\lambda+j,\beta+\lambda}(x).
\end{equation}
We have that $u_j^{\lambda}(x)\in\mathbb{R}[\theta_x]$, where $\theta_x=x(x+\alpha+\beta +1)$ (see \cite[p. 216]{ddI3}). Take now $\alpha $ and $\beta$ two positive integers with $m_2\le \alpha\le \max H$, $m_1\le\beta\le\max G$,
$$
G=\{\beta,\beta+1,\ldots ,\beta+m_1-1\},\quad H=\{\alpha,\alpha+1,\ldots ,\alpha+m_2-1\},
$$
and
\begin{align}\label{pel1}
\R_{g_k}(\theta_x)&=u_{\beta+k-1}^\alpha(x)+\sum_{l=0}^{k-1}\frac{(\beta+k-l)_l\binom{k-1}{l}a_{k-l-1}}{(-1)^l(\beta-l)_l}
u_l^\alpha(x),\quad k=1,\ldots,m_1,\\\label{pel2}
\Ss_{h_k}(\theta_x)&=u_{\alpha+k-1}^\alpha(x)+\sum_{l=0}^{k-1}\frac{(\alpha+k-l)_l
\binom{k-1}{l}b_{k-l-1}}{(-1)^l(\alpha-l)_l}
u_l^\alpha(x),\quad k=1,\ldots,m_2,
\end{align}
where $a_k$, $k=0,\ldots, m_1-1$, $b_k$, $k=0,\ldots, m_2-1$, are real numbers with $a_0,b_0\not=0$. Then the polynomials (\ref{iquss}) are orthogonal with respect to the Krall-Jacobi weight (\ref{kjm3}) (for certain parameters $c_k$, $k=0,\ldots, m_2-1$, $d_k$, $k=0,\ldots, m_1-1$). In \cite[(1.13) and Example 4.1, 1, p. 217]{ddI3}, we represent $(q_n)_n$ with a different set of polynomials $(\mathcal{Z}_l)_{l=1}^m$ (\ref{defpol}) from where the representation $\{\R_g\}_{g\in G}$, $\{\Ss_h\}_{h\in H}$ can be easily obtained.

\medskip

As the main results of this paper, we first prove that for any set of polynomials $\mathcal{R}_g,g\in G$, with $\deg\mathcal{R}_g=g$, and $\mathcal{S}_h,h\in H$, with $\deg\mathcal{S}_h=h$, satisfying (\ref{assum}), the polynomials $(q_n)_n$ (\ref{iquss}) are bispectral. And second, we also prove that the only sequences $(q_n)_n$ (\ref{iquss}) satisfying a three-term recurrence relation (and therefore they are orthogonal with respect to a measure) are essentially the Krall-Jacobi polynomials orthogonal with respect to the any of the measures (\ref{kjm1}), (\ref{kjm2}) or (\ref{kjm3}).

\medskip

The content of the paper is as follows. After some preliminaries in Section 2, in Section 3 we find some orthogonality properties for the polynomials $(q_n)_n$ with respect to a certain bilinear form. When $\alpha-\max G\not =0,-1,-2,\ldots $ and $\beta-\max H\not =0,-1,-2,\ldots $, we get this bilinear form by modifying the Jacobi weight with a nonsymmetric perturbation (which strongly depends on the polynomials $\Rc_g$ and $\mathcal{S}_h$). When $1\le \alpha \le \max G$ and/or $1\le \beta \le \max H$ (which includes the Krall-Jacobi polynomials orthogonal with respect to (\ref{kjm1}), (\ref{kjm2}) or (\ref{kjm3})), in order to get orthogonality properties, we have to transform a portion of that perturbation into a discrete Sobolev part.

These orthogonality properties allow us to prove in Section 4 that the sequence $(q_n)_n$ satisfies some recurrence relations of the form (\ref{qrr}) where $s=-r$. On the other hand, the orthogonality properties constrain the number of terms of these recurrence relations: in particular, we prove in Section 4 that when $\alpha-\max G\not =0,-1,-2,\ldots $ and $\beta-\max H\not =0,-1,-2,\ldots $, the sequence $(q_n)_n$ can never satisfy a three-term recurrence relation of the form (\ref{ttrr}). When $1\le \alpha \le \max G$ and/or $1\le \beta \le \max H$, we prove that the sequence $(q_n)_n$ satisfies a three-term recurrence relation of the form (\ref{ttrr}) only when they correspond with the Krall-Jacobi cases (\ref{kjm1}), (\ref{kjm2}) or (\ref{kjm3}).

We also prove some results for the algebra of operators $\al_n$, defined as follows. We denote by $\A_n$ the algebra formed by all higher-order difference operators (acting on the variable $n$) of the form (\ref{doho2g}). Then we define
\begin{equation*}\label{algo}
\al _n=\{D_n\in \A_n: D_n(q_n)=Q(x)q_n,\; Q\in\RR[x]\},
\end{equation*}
where $\RR[x]$ denotes the linear space of real polynomials in the unknown $x$. This algebra is actually characterized by the algebra of polynomials defined from the corresponding eigenvalues
\begin{equation*}\label{algp}
\alp _n=\{Q\in\RR[x]: \mbox{there exists $D_n\in \al_n$ such that $D_n(q_n)=Q(x)q_n$}\}.
\end{equation*}
In Section 4 we prove that when $\alpha-\max G\not =0,-1,-2,\ldots $ and $\beta-\max H\not =0,-1,-2,\ldots $ and $G$ is a segment, i.e. its elements are consecutive positive integers, the algebra $\alp_n$ has a simple estructure:
$$
\alp_n=\{Q\in\RR[x]: \mbox{$(1+x)^{\max G}(1-x)^{\max H}$ divides $Q'$}\}.
$$
For a characterization of the corresponding algebra for the Charlier and Meixner type polynomials see \cite{ducb, DuR}. We also give some examples showing that, in general, this algebra can have a more complicated structure.

The structure of the Jacobi case is technically more involved than that of the Laguerre case studied in \cite{ddIlb}. On one hand, we have to use a more complicated basis $(b_s)_{s\ge 1}$ in the linear space of polynomials (see (\ref{basisb}) below). On the other hand, we have to work with a pair of finite sets of positive integers instead of only one set, and more parameters (in any case we will omit those proofs which are too similar to the corresponding ones in \cite{ddIlb} for the Laguerre type polynomials).

\section{Preliminaries}

Consider two finite sets $G=\{g_1,\ldots,g_{m_1}\}$ and $H=\{h_1,\ldots,h_{m_2}\}$ of positive integers (written in increasing size) and polynomials $\mathcal{R}_g,g\in G$, with $\deg\mathcal{R}_g=g$ and $\mathcal{S}_h,h\in H$, with $\deg\mathcal{S}_h=h$. We associate to $G$ and $H$ the sequence of polynomials $(q_n)_n$ defined by \eqref{iquss}. Along this paper we will always assume that $\Lambda_{G,H}(n)\neq 0, n\geq0$.

\begin{remark}\label{qsc}
When $\alpha$ and $\beta$ are integers, $\mathfrak p(n)\mathfrak q(n)$ can vanish for some $n=0,\ldots, m-1$, where $\mathfrak p$ and $\mathfrak q$ are the polynomials defined by (\ref{def1pi}) and (\ref{def1qi}), respectively. However, even if for some $n=0,\ldots, m-1$, $\mathfrak p(n)\mathfrak q(n)=0$, the ratio $\Lambda_{G,H}$ (\ref{casdet}) and the polynomial $q_n$ (\ref{iquss}) are well-defined (and hence $q_n$ has degree $n$ if and only if $\Lambda_{G,H}(n)\not =0, n\geq0$). This can be proved as for the Jacobi-Sobolev polynomials studied in \cite[p. 205]{ddI3}.
\end{remark}

We will use the following alternative definition of the polynomials $(q_n)_n$ in \eqref{iquss}. For $j=0,1,\ldots,m$, let us define the sequences $(\beta_{n,j})_n$ by
\begin{equation}\label{betaj}
\beta_{n,j}=\frac{1}{\mathfrak p(n)\mathfrak q(n)}\mbox{det}\left(\rho_{n,i}^l\mathcal{Z}_l(\theta_{n-i})\right)_{\begin{subarray}{1} l=1,\ldots,m,\\
i=0,\ldots, m,i\not =j\end{subarray}}.
\end{equation}
By expanding the determinant \eqref{iquss} by its first row (writing $J_{u}^{\alpha,\beta}(x)=0$ for $u<0$) we get the expansion
\begin{equation}\label{qqbeta}
q_n(x)=\sum_{j=0}^{m\wedge n}\beta_{n,j}J_{n-j}^{\alpha,\beta}(x).
\end{equation}
A straightforward computation, using \eqref{casdet}, \eqref{def1pi}, \eqref{def1qi}, \eqref{betaj} and \eqref{assum}, shows that
\begin{equation}\label{betajm}
\beta_{n,m}=(-1)^{m_1}\left(\frac{n+\alpha-m+1}{n+\beta}\right)^{m_1}\frac{\mathfrak p(n+1)\mathfrak q(n+1)}{\mathfrak p(n)\mathfrak q(n)}\Lambda_{G,H}(n+1),
\end{equation}
where the polynomials $\mathfrak p$ and $\mathfrak q$ are defined by (\ref{def1pi}) and (\ref{def1qi}), respectively.


On the other hand, substituting the first row in \eqref{iquss} by any other row in that determinant, we get the trivial identity
\begin{equation}\label{vanisz}
\sum_{j=0}^m(-1)^j\beta_{n,j}\rho_{n,j}^l\mathcal{Z}_l(\theta_{n-j})=0,\quad l=1,\ldots,m.
\end{equation}
\begin{remark}\label{r2.1}
We stress that if we substitute the polynomials $\Rc_g$ and $\mathcal{S}_h$  in the determinant (\ref{iquss}) by any linear combination $R_g$ and $S_h$ of the form
\begin{equation*}\label{ocl1}
R_g=\Rc_g+\sum_{\tilde g\in G;\tilde g<g}\zeta_{g,\tilde g}\Rc_{\tilde g},\quad
S_h=\mathcal{S}_h+\sum_{\tilde h\in H;\tilde h<h}\chi_{h,\tilde h}\mathcal{S}_{\tilde h},
\end{equation*}
then the polynomials $(q_n)_n$ remain invariant. Notice that $\deg R_g=\deg \Rc_g=g$, $\deg S_h=\deg \mathcal{S}_h=h$ and $R_g$ and $S_h$ are again monic polynomials.
\end{remark}
Given polynomials $Y_i$, $i=0,\ldots, m$, with $\deg Y_i=u_i$, we write
\begin{equation}\label{iqussy}
W_a^Y(x)=\frac{\begin{vmatrix}
              \no^1_{x,0}Y_0(\theta_{x}) &\no^1_{x,1} Y_0(\theta_{x-1}) & \cdots &  \no^1_{x,m}Y_0(\theta_{x-m}) \\
              \no^1_{x,0}Y_1(\theta_{x}) &\no^1_{x,1} Y_1(\theta_{x-1}) & \cdots &  \no^1_{x,m}Y_1(\theta_{x-m}) \\
               \vdots & \vdots & \ddots & \vdots \\
\no^m_{x,0}Y_m(\theta_{x}) & \no^m_{x,1} Y_m(\theta_{x-1}) & \cdots & \no^m_{x,m}Y_m(\theta_{x-m})
             \end{vmatrix}}{\mathfrak p(x)\mathfrak q(x)}.
\end{equation}
Using Lemma A.1 of \cite{ddI3}, it follows easily that when $u_i\neq u_j$, $0\le i,j\le m_1$, $i\neq j$, and
$m_1+1\le i,j\le m$, $i\neq j$, then $W_a^Y(x)$ is a polynomial in $x$ of degree
\begin{equation}\label{gcd}
d=2\left[\sum_{i=0}^mu_i-\binom{m_1+1}{2}-\binom{m_2}{2}\right].
\end{equation}
Otherwise, $W_a^Y(x)$ is a polynomial in $x$ of degree strictly less than $d$.

Analogously, given polynomials $Y_i$, $i=1,\ldots, m+1$, with $\deg Y_i=u_i$, the function
\begin{equation}\label{iqussy2}
W_b^Y(x)=\frac{\begin{vmatrix}
              \no^1_{x,0}Y_1(\theta_{x}) &\no^1_{x,1} Y_1(\theta_{x-1}) & \cdots &  \no^1_{x,m}Y_1(\theta_{x-m}) \\
                            \vdots & \vdots & \ddots & \vdots \\
\no^m_{x,0}Y_m(\theta_{x}) & \no^m_{x,1} Y_m(\theta_{x-1}) & \cdots & \no^m_{x,m}Y_m(\theta_{x-m})\\
 \no^m_{x,0}Y_{m+1}(\theta_{x}) &\no^m_{x,1} Y_{m+1}(\theta_{x-1}) & \cdots &  \no^m_{x,m}Y_{m+1}(\theta_{x-m}) \\
             \end{vmatrix}}{\mathfrak p(x)\mathfrak q(x)},
\end{equation}
is a polynomial in $x$ of degree
$$
d=2\left[\sum_{i=1}^{m+1}u_i-\binom{m_1}{2}-\binom{m_2+1}{2}\right],
$$
if and only if $u_i\neq u_j$, $1\le i,j\le m_1$, $i\neq j$, and
$m_1+1\le i,j\le m+1$, $i\neq j$. Otherwise, $W_b^Y(x)$ is a polynomial in $x$ of degree strictly less than $d$.

We will also need the following combinatorial formula: if $\alpha,\beta, s, k, u$ are nonnegative integers with $s\geq\beta+k$, then
\begin{equation}\label{formulaca}
\sum_{j=0}^{s-\beta}\binom{s-\beta-k}{j-k}\binom{u+\alpha+\beta-k+j}{\beta-k+j}\binom{u+\alpha+\beta-s+k}{\alpha+\beta-s+j}=\binom{u+\alpha+\beta}{\alpha}\binom{u+s-k}{s-k}.
\end{equation}

\medskip

Part of the difficulties in the Jacobi case (compared with the Laguerre case) comes from the fact that we have to work with the following basis of the linear space of polynomials $\mathbb{R}[x]$ (instead of the usual basis of monomials):
\begin{equation}\label{basisb}
b_s(x)=\begin{cases}(1+x)^{s-1}(1-x)^{m_2},&\mbox{for $s=1,\ldots, m_1$,}\\
(1+x)^{m_1}(1-x)^{s-m_1-1},&\mbox{for $s=m_1+1,\ldots ,m$,}\\
(1+x)^{m_1}(1-x)^{m_2}x^{s-m-1},&\mbox{for $s=m+1,\ldots$}
\end{cases}
\end{equation}
Let us write $\gamma_{s}^i$ as the coefficients of the change of basis $(b_s(x))_s\to (x^{s-1})_{s-1}$, $s\geq1$:
\begin{equation}\label{qnibi}
x^i=\sum_{s=1}^{(i+1)\vee m}\gamma_s^ib_s(x),\quad i\geq0.
\end{equation}

\begin{lemma}\label{nwl}
The coefficients $\gamma_s^i$ can be recursively obtained from the following relations:
\begin{align*}
\sum_{h=1}^l(-2)^{h-l}\binom{m_2}{l-h}\gamma_{h}^i&=\frac{(-1)^{i+l+1}}{2^{m_2}}\binom{i}{l-1},\quad l=1,\ldots,m_1,\\
\sum_{h=1}^l(-2)^{h-l}\binom{m_1}{l-h}\gamma_{m_1+h}^i&=\frac{(-1)^{l+1}}{2^{m_1}}\binom{i}{l-1},\quad l=1,\ldots,m_2.
\end{align*}
\end{lemma}

\begin{proof}
It is a matter of computation, by taking derivatives in the expansion (\ref{qnibi}), evaluating at $x=1,-1,$ and using the following formulas: for $s=1,\ldots,m_1,$
\begin{equation}\label{propb1}
\left(b_s\right)^{(l)}(-1)=\begin{cases}l!2^{m_2}(-2)^{s-l-1}\displaystyle\binom{m_2}{l-s+1}&\mbox{for $l=s-1,\ldots, [m_1\wedge(s+m_2)]-1$,}\\
0,&\mbox{for $l=0,\ldots ,s-2$, or $l=s+m_2,\ldots, m_1-1$},
\end{cases}
\end{equation}
while for $s\geq m_1+1$ we have $\left(b_s\right)^{(l)}(-1)=0$ for $l=0,\ldots,m_1-1$.

And, similarly, for $s=m_1+1,\ldots,m,$ we have
$$
\left(b_s\right)^{(l)}(1)=\begin{cases}l!(-1)^{s-m_1-1}2^{s-l-1}\displaystyle\binom{m_1}{s-l-1}&\mbox{for $l=s-m_1-1,\ldots, [s\wedge m_2]-1$,}\\
0,&\mbox{for $l=0,\ldots ,s-m_1-2$, or $l=s,\ldots,m_2-1$,}
\end{cases}
$$
while for $s=1,\ldots, m_1,$ or $s\geq m+1$ we have $\left(b_s\right)^{(l)}(1)=0$ for $l=0,\ldots m_2-1$ (see proof of Lemma 2.1 in \cite{ddI3}).
\end{proof}

We also need the following technical lemma.

\begin{lemma}\label{fisi}
For $s=0,1,\ldots,m_1-1,$ and $t=0,1,\ldots,m_2-1,$ define the rational functions
$$
\phi_s(x)=\sum_{u=0}^{m_1-s-1}a_u^s(1+x)^{-u-1},\quad\psi_t(x)=\sum_{u=0}^{m_1-t-1}c_u^t(1-x)^{-u-1}.
$$
Then there exist numbers $a_u^s$ and $c_u^t$ (uniquely determined) such that for $i=1,\ldots,m_1,$ we have
\begin{equation}\label{relfi}
\sum_{l=i-1}^{[m_1\wedge(i+m_2)]-1}(-2)^{i-l-1}\binom{m_2}{l-i+1}\phi_l(x)=-(1+x)^{i-m_1-1},
\end{equation}
and for $i=m_1+1,\ldots,m,$ we have
\begin{equation*}\label{relsi}
\sum_{l=i-m_1-1}^{[m_2\wedge i]-1}(-2)^{i-l-1}\binom{m_1}{i-l-1}\psi_l(x)=-(1-x)^{i-m-1}.
\end{equation*}
\end{lemma}
\begin{proof}
The rational functions $\phi_s(x)$ can be recursively computed starting from $i=m_1$ in which case we have $\phi_{m_1-1}(x)=-(1-x)^{-1}$. From here we can compute $\phi_{m_1-2}(x),\ldots,\phi_0(x)$ recursively. The same can be applied to $\psi_t(x)$.
\end{proof}

\medskip

As we wrote in the Introduction, the most interesting case corresponds with the different families of Krall-Jacobi polynomials, orthogonal with respect to any of the measures (\ref{kjm1}), (\ref{kjm2}) or (\ref{kjm3}). In any of these cases we have $m_2\le \alpha\le \max H$ or/and $m_1\le\beta\le\max G$.
For instance, the orthogonal polynomials with respect to the measure (\ref{kjm3}) are the polynomials $(q_n)_n$ (\ref{iquss}), where $m_2\le \alpha\le \max H$, $m_1\le\beta\le\max G$,
$$
G=\{\beta,\beta+1,\ldots ,\beta+m_1-1\},\quad H=\{\alpha,\alpha+1,\ldots ,\alpha+m_2-1\},
$$
and the polynomials $\R_{g}$, $g\in G$, $\Ss_h$, $h\in H$, are defined by (\ref{pel1}) and (\ref{pel2}), respectively.

If $1\le \alpha\le m_2-1$ and/or $1\le \beta\le m_1-1$ and if we take
\begin{align*}
\R_{g_k}(\theta_x)&=u_{\beta+k-1}^\alpha(x)+\sum_{l=0}^{k+\beta-m_1-1}\frac{(\beta+k-l)_l\binom{k-1}{l}\tilde a_{k-l-1}}{(-1)^l(\beta-l)_l}
u_l^\alpha(x),\quad k=1,\ldots,m_1,\\
\Ss_{h_k}(\theta_x)&=u_{\alpha+k-1}^\alpha(x)+\sum_{l=0}^{k+\alpha-m_2+1}\frac{(\alpha+k-l)_l
\binom{k-1}{l}\tilde b_{k-l-1}}{(-1)^l(\alpha-l)_l}
u_l^\alpha(x),\quad k=1,\ldots,m_2,
\end{align*}
then the polynomials $(q_n)_n$ satisfies also three-term recurrence relations (if $u>v$ we always take $\sum_{l=u}^{v}\rho_l=0$). But these polynomials $(q_n)_n$, as in the case of Laguerre type polynomials (see \cite{ ddIlb}), are somehow degenerated. Indeed, since we have (after a combination of formulas (3.94), (3.100) and (3.107) of \cite{STW}, see also \cite[p. 203]{ddI3})
\begin{align}
\label{fff1}\left((1+x)^kJ_n^{\alpha,\beta}\right)^{(j)}(-1)&=\frac{(-1)^{j+k}j!}{2^{j-k}\binom{\alpha+\beta}{\beta}}\binom{n+\alpha+\beta}{\alpha}\binom{n+\beta}{n-j+k}\binom{n+\alpha+\beta+j-k}{j-k},\\
\nonumber\left((1-x)^kJ_n^{\alpha,\beta}\right)^{(j)}(1)&=\frac{(-1)^{n+k}j!}{2^{j-k}\binom{\alpha+\beta}{\beta}}\binom{n+\alpha+\beta}{\alpha}\binom{n+\alpha}{n-j+k}\binom{n+\alpha+\beta+j-k}{j-k},
\end{align}
it is not difficult to see that
\begin{align*}
q_n^{(j)}(1)=0,&\quad j=0,1,\ldots,m_2-\alpha-1,\quad n\geq m-\alpha-\beta,\\
q_n^{(j)}(-1)=0,&\quad j=0,1,\ldots,m_1-\beta-1,\quad n\geq m-\alpha-\beta.
\end{align*}
Moreover, it turns out that the polynomials
$$
\left(\frac{q_{n+m-\alpha-\beta}(x)}{(1-x)^{m_2-\alpha}(1+x)^{m_1-\beta}}\right)_{n\ge 0},
$$
are particular examples of the cases (\ref{pel1}) and (\ref{pel2}) for certain choice of the parameters involved.

\section{Orthogonality properties}\label{jorth}
In this section we will establish some orthogonality properties for the polynomials $(q_n)_n$ in \eqref{iquss}. We start with the case $\alpha-\max H\neq 0, -1, -2, \ldots,$ and $\beta-\max G\neq 0,-1, -2, \ldots$. Using the polynomials $u_j^{\lambda}(x)$ (\ref{basu}), we can always write
\begin{equation}\label{defrs}
\mathcal{R}_g(\theta_x)=\sum_{s=0}^g\nu_s^gu_s^{\alpha}(x),\quad \mathcal{S}_h(\theta_x)=\sum_{s=0}^h\omega_s^hu_s^{\alpha}(x),
\end{equation}
for certain numbers $\nu_s^g,s=0,1,\ldots,g,$ and $\omega_s^h,s=0,1,\ldots,h,$ ($\nu_g^g=1, \omega_h^h=1$), which are uniquely determined from the polynomials $\mathcal{R}_g$ and $\mathcal{S}_h$.

For each $l=0,1,\ldots m_1-1$, we introduce the rational function $U_l$  defined by
\begin{equation}
\label{Uig}U_l(x)=\phi_l(x)+\frac{\kappa_l}{\Gamma(\beta)}\sum_{s=0}^{g_{l+1}}
\frac{(\beta-s)_s2^ss!\nu_s^{g_{l+1}}}{(1+x)^{s+1}},
\end{equation}
where the rational functions $\phi_l(x), l=0,1,\ldots m_1-1,$ are defined in Lemma \ref{fisi} and $\kappa_l$, $l=0,\ldots , m_1-1$, are real numbers.

Similarly, for each $l=0,1,\ldots m_2-1$, the rational function $V_l$ is defined by
\begin{equation}
\label{Vig}V_l(x)=\psi_l(x)+
\frac{\tau _l}{\Gamma(\alpha)}\sum_{s=0}^{h_{l+1}}
\frac{(\alpha-s)_s2^ss!\omega_s^{h_{l+1}}}{(1-x)^{s+1}},
\end{equation}
where the rational functions $\psi_l(x), l=0,1,\ldots m_2-1,$ are defined in Lemma \ref{fisi} and $\tau_l$, $l=0,\ldots , m_2-1$, are real numbers.

Let us consider now the following bilinear form, with respect to which the polynomials $(q_n)_{n\ge m}$ in \eqref{iquss} will be (left) orthogonal. For real numbers $\alpha$ and $\beta$ such that $\alpha-\max H\neq 0, -1, -2, \ldots,$ and $\beta-\max G\neq 0,-1, -2, \ldots,$ we consider
\begin{equation}\label{inner}
\langle p,q\rangle=\langle p,q\rangle_1+\langle p,q\rangle_2+\langle p,q\rangle_3,
\end{equation}
where
\begin{equation*}\label{inner1}
\langle p,q\rangle_1=\int_{-1}^1p(x)q(x)\mu_{\alpha-m_2,\beta-m_1}(x)dx,
\end{equation*}
and $\mu_{\alpha,\beta}$ is the Jacobi weight given by \eqref{pJac},
\begin{equation}\label{inner2}
\langle p,q\rangle_2=\sum_{l=0}^{m_1-1}\frac{q^{(l)}(-1)}{2^{m_2}l!}\int_{-1}^1p(x)U_l(x)\mu_{\alpha,\beta}(x)dx,
\end{equation}
where $U_l$ is defined by \eqref{Uig}, and
\begin{equation}\label{inner3}
\langle p,q\rangle_3=\sum_{l=0}^{m_2-1}\frac{q^{(l)}(1)}{(-1)^{m_1+l}l!}\int_{-1}^1p(x)V_l(x)\mu_{\alpha,\beta}(x)dx.
\end{equation}
where $V_l$ is defined by \eqref{Vig}.

The following lemma will be the key for most of our results.
\begin{lemma}\label{lemwenok}
Let $k,n\geq0$, $j=0,1,\ldots,n$ and $n-j\ge k$. For $i=1,\ldots,m_1,$ we have
$$
\langle (1+x)^kJ_{n-j}^{\alpha,\beta},b_i\rangle=\frac{c_{n,i}\rho_{n,j}^i}{(-1)^j}\left[\sum_{l=i-1}^{[m_1\wedge (m_2+i)]-1}\frac{\kappa_l\binom{m_2}{l-i+1}}{(-2)^{l+1}} \sum_{s=k}^{g_{l+1}}2^k\nu_s^{g_{l+1}}(\beta-s)_k(s-k+1)_ku_{s-k}^{\alpha}(n-j)\right],
$$
where $c_{n,i}=(-1)^{m+i}2^{\alpha+\beta+i}
\Gamma(\beta+1)\Gamma(n+\alpha-m+1)/(\Gamma(\alpha+\beta+1)\Gamma(n+\beta))$ and $\rho_{x,j}^i$ is defined by \eqref{defno}. Similarly, for $i=m_1+1,\ldots,m,$ we have
$$
\langle (1-x)^kJ_{n-j}^{\alpha,\beta},b_i\rangle=\frac{d_{n,i}}{(-1)^j}\left[\sum_{l=i-m_1-1}^{[m_2\wedge i]-1}\frac{\tau_l\binom{m_1}{i-l-1}}{(-2)^{l+1}}\sum_{s=k}^{h_{l+1}}2^k\omega_s^{h_{l+1}}(\alpha-s)_k(s-k+1)_k
u_{s-k}^{\alpha}(n-j)\right],
$$
where $d_{n,i}=(-1)^{n+i}2^{\alpha+\beta+i}\Gamma(\beta+1)/\Gamma(\alpha+\beta+1)$.
\end{lemma}

\begin{proof}
We will need the following formula, which can be found in \cite[p. 203]{ddI3}: for $l=1,\ldots,m_1$, we have
\begin{equation}\label{formJ1}
\int_{-1}^1J_n^{\alpha,\beta}(x)\mu_{\alpha,\beta-l}(x)dx=
\frac{\Gamma(\beta+1)\Gamma(\beta-l+1)}{2^{l-\alpha-\beta-1}\Gamma(\alpha+\beta+1)}
\binom{n+l-1}{l-1}\Gamma_{\beta,\alpha+\beta-l+1}^{\alpha,\alpha+\beta}(n).
\end{equation}
Now, for $i=1,\ldots,m_1,$ we have, using (\ref{propb1}), that $\langle (1+x)^kJ_{n-j}^{\alpha,\beta},b_i\rangle_3=0$. Using the definition of $b_i(x)$ and $U_l$ in \eqref{basisb} and \eqref{Uig}, respectively, and (\ref{propb1}) again, we have
\begin{align*}
\langle & (1+x)^kJ_{n-j}^{\alpha,\beta},b_i\rangle=\int_{-1}^1J_{n-j}^{\alpha,\beta}\mu_{\alpha,\beta+k+i-m_1-1}dx\\
&\quad+\sum_{l=i-1}^{[m_1\wedge(m_2+i)]-1}\frac{\binom{m_2}{l-i+1}}{(-2)^{l-i+1}}\int_{-1}^1J_{n-j}^{\alpha,\beta}\left(\phi_l(x)+\frac{\kappa_l}{\Gamma(\beta)}\sum_{s=0}^{g_{l+1}}\frac{(\beta-s)_s2^ss!\nu_s^{g_{l+1}}}{(1+x)^{s+1}}\right)\mu_{\alpha,\beta+k}dx\\
&=\sum_{l=i-1}^{[m_1\wedge(m_2+i)]-1}\frac{\kappa_l\binom{m_2}{l-i+1}}{(-2)^{l-i+1}\Gamma(\beta)}\sum_{s=0}^{g_{l+1}}(\beta-s)_s2^ss!\nu_s^{g_{l+1}}\int_{-1}^1J_{n-j}^{\alpha,\beta}\mu_{\alpha,\beta+k-s-1}dx\\
&=\sum_{l=i-1}^{[m_1\wedge(m_2+i)]-1}\frac{\kappa_l\binom{m_2}{l-i+1}}{(-2)^{l-i+1}\Gamma(\beta)}\times\\
&\qquad\times\sum_{s=0}^{g_{l+1}}(\beta-s)_s2^ss! \nu_s^{g_{l+1}}\left[\frac{\Gamma(\beta+1)\Gamma(\beta-s+k)}{2^{s-k-\alpha-\beta}\Gamma(\alpha+\beta+1)}\binom{n+s-k}{s-k}\Gamma_{\beta,\alpha+\beta-s+k}^{\alpha,\alpha+\beta}(n-j)\right]\\
&=c_{n,i}(-1)^j\rho_{n,j}^i\left[\sum_{l=i-1}^{[m_1\wedge(m_2+i)]-1}\frac{\kappa_l\binom{m_2}{l-i+1}}{(-2)^{l+1}}\sum_{s=k}^{g_{l+1}}2^k\nu_s^{g_{l+1}}(\beta-s)_k(s-k+1)_ku_{s-k}^{\alpha}(n-j)\right].
\end{align*}
The second step follows as a consequence of Lemma \ref{fisi} (see \eqref{relfi}), while the third step follows from \eqref{formJ1} and the definition of $\rho_{n,j}^i$ in \eqref{defno} (see also \eqref{gammas}) and the definition of $u_j^\lambda(x)$ in \eqref{basu}. The second identity can be proved in a similar way.
\end{proof}

\begin{remark}\label{remk0}
Observe that for $k=0$ we have, using \eqref{defrs}, the simplified formulas
\begin{equation*}
\langle J_{n-j}^{\alpha,\beta},b_i\rangle=c_{n,i}(-1)^j\rho_{n,j}^i
\sum_{l=i-1}^{[m_1\wedge(m_2+i)]-1}\frac{\kappa_l\binom{m_2}{l-i+1}}{(-2)^{l+1}}\mathcal{R}_{g_{l+1}}(\theta_{n-j}),
\end{equation*}
for $i=1,\ldots,m_1,$ and
\begin{equation*}
\langle J_{n-j}^{\alpha,\beta},b_i\rangle=d_{n,i}(-1)^j\sum_{l=i-m_1-1}^{[m_2\wedge i]-1}\frac{\tau_l\binom{m_1}{i-l-1}}{(-2)^{l+1}}\mathcal{S}_{h_{l+1}}(\theta_{n-j}),
\end{equation*}
for $i=m_1+1,\ldots,m,$ using the notation in Lemma \ref{lemwenok}.
\end{remark}

We are now ready to prove the orthogonality properties for the polynomials $(q_n)_n$.

\begin{theorem}\label{thoj}
Let $\alpha$ and $\beta$ be real numbers with $\alpha-\max H\neq 0, -1, -2, \ldots,$ and $\beta-\max G\neq 0,-1, -2, \ldots$. Assume that conditions \eqref{assum} hold. Then  for $n\geq m$, the polynomials $(q_n)_n$ defined by \eqref{iquss} satisfy the following orthogonality properties with respect to the bilinear from \eqref{inner}:
\begin{align*}
\langle q_n,q_i\rangle&=0,\quad i=0,1,\ldots,n-1,\\
\langle q_n,q_n\rangle&\neq0.
\end{align*}
\end{theorem}
\begin{proof}
For $n\geq m$, let $q_n$ be written as in \eqref{qqbeta} and $x^i$ as in \eqref{qnibi}. Then, using Lemma \ref{lemwenok} for $k=0$ (see Remark \ref{remk0}), we have
\begin{align*}
\langle q_n,x^i\rangle&=\sum_{j=0}^{m\wedge n}\beta_{n,j}\sum_{s=1}^{(i+1)\vee m}\gamma_s^i\langle J_{n-j}^{\alpha,\beta}, b_s\rangle\\
&=\sum_{j=0}^{m\wedge n}\beta_{n,j}\left(\sum_{s=1}^{m_1}\gamma_s^ic_{n,s}(-1)^j\rho_{n,j}^s
\left[\sum_{l=s-1}^{[m_1\wedge(m_2+s)]-1}\frac{\kappa_l\binom{m_2}{l-s+1}}{(-2)^{l+1}}\mathcal{R}_{g_{l+1}}(\theta_{n-j})\right]\right.\\
&\qquad+\sum_{s=m_1+1}^m\gamma_s^id_{n,s}(-1)^j\left[\sum_{l=s-m_1-1}^{[m_2\wedge s]-1}\frac{\tau_l\binom{m_1}{s-l-1}}{(-2)^{l+1}}\mathcal{S}_{h_{l+1}}(\theta_{n-j})\right]\\
&\left.\qquad+\sum_{s=m+1}^{(i+1)\vee m}\gamma_s^i\int_{-1}^1J_{n-j}^{\alpha,\beta}(x)x^{s-m-1}\mu_{\alpha,\beta}dx\right)\\
&=\sum_{s=1}^{m_1}\gamma_s^ic_{n,s}\sum_{l=s-1}^{[m_1\wedge(m_2+s)]-1}\frac{\kappa_l\binom{m_2}{l-s+1}}{(-2)^{l+1}}\left(\sum_{j=0}^{m\wedge n}\beta_{n,j}(-1)^j\rho_{n,j}^s\mathcal{R}_{g_{l+1}}(\theta_{n-j})\right)\\
&\qquad+\sum_{s=m_1+1}^m\gamma_s^id_{n,s}\sum_{l=s-m_1-1}^{[m_2\wedge s]-1}\frac{\tau_l\binom{m_1}{s-l-1}}{(-2)^{l+1}}\left(\sum_{j=0}^{m\wedge n}\beta_{n,j}(-1)^j\mathcal{S}_{h_{l+1}}(\theta_{n-j})\right)\\
&\qquad+\sum_{j=0}^{m\wedge n}\beta_{n,j}\sum_{s=m+1}^{(i+1)\vee m}\gamma_s^i\int_{-1}^1J_{n-j}^{\alpha,\beta}(x)x^{s-m-1}\mu_{\alpha,\beta}dx.
\end{align*}
Assume first that $m\le i\le n-1$. In this case the third addend in the previous formula vanishes because of the orthogonality of the Jacobi polynomials. Indeed, the power of $x^{s-m-1}$ goes from $0$ to $i-m$ (since $i\geq m$) and we always have in this situation that $i-m<n-j$ for $j=0,\ldots,m\wedge n=m$. The other two addends also vanish as a consequence of \eqref{vanisz}, since in this case we have $m\wedge n=m$. For $i=n\geq m$ we have $\langle q_n,x^n\rangle\neq0$ as a consequence of the orthogonality properties for the Jacobi polynomials and that $\beta_{m,m}\not =0$ (see \eqref{betajm}).

Let us now assume that $0\leq i\leq m-1$. In this case there is no third addend in the previous formula (since $(i+1)\vee m<m+1$). Therefore we have
\begin{align}\nonumber
\langle q_n,x^i\rangle&=\sum_{s=1}^{m_1}\gamma_s^ic_{n,s}\sum_{l=s-1}^{[m_1\wedge(m_2+s)]-1}\frac{\kappa_l\binom{m_2}{l-s+1}}{(-2)^{l+1}}\left(\sum_{j=0}^{m\wedge n}\beta_{n,j}(-1)^j\rho_{n,j}^s\mathcal{R}_{g_{l+1}}(\theta_{n-j})\right)\\\label{npf}
&\qquad+\sum_{s=m_1+1}^m\gamma_s^id_{n,s}\sum_{l=s-m_1-1}^{[m_2\wedge s]-1}\frac{\tau_l\binom{m_1}{s-l-1}}{(-2)^{l+1}}\left(\sum_{j=0}^{m\wedge n}\beta_{n,j}(-1)^j\mathcal{S}_{h_{l+1}}(\theta_{n-j})\right).
\end{align}
For $n\geq m$ we have $m\wedge n=m$ and, as before, the two first addends vanish as a consequence of \eqref{vanisz}.
\end{proof}

Notice that, we have not assumed any constrain on the real numbers $\kappa_l$, $l=0,\ldots,m_1-1,$ and  $\tau_l$, $l=0,\ldots,m_2-1$. In general, the orthogonality properties in Theorem \ref{thoj} can not be extended for $n=0,\ldots , m-1$. This fact would imply some problems to prove the bispectrality of the polynomials $(q_n)_n$. We will avoid those problems by using the following lemma.

\begin{lemma}\label{cdsm} Let $\alpha$ and $\beta$ be real numbers with $\alpha-\max H\neq 0, -1, -2, \ldots,$ and $\beta-\max G\neq 0,-1, -2, \ldots$. Assume that conditions \eqref{assum} hold and that $\kappa_l =1$, $l=0,\ldots,m_1-1$ and  $\tau_l=1$, $l=0,\ldots,m_2-1$. Then the matrix $A$ with entries $A_{i,j}=\langle q_j,x^i\rangle $, $i,j=0,\ldots, m-1$, is nonsingular.
\end{lemma}

\begin{proof}
We will prove that $\det A\not =0$. First we will find an explicit expression for $\langle q_j,x^i\rangle $, $i,j=0,\ldots, m-1$.
For $0\leq j\leq m-1$, using \eqref{vanisz}, we can write
\begin{align*}
\sum_{l=0}^j(-1)^l\beta_{j,l}\rho_{j,l}^s\mathcal{Z}_s(\theta_{j-l})&=-\sum_{l=j+1}^m(-1)^l\beta_{j,l}\rho_{j,l}^s\mathcal{Z}_s(\theta_{j-l})\\
&=-\sum_{l=1}^{m-j}(-1)^{j+l}\beta_{j,j+l}\rho_{j,j+l}^s\mathcal{Z}_s(\theta_{-l}).
\end{align*}
Substituting this expression in formula (\ref{npf}), using that $m\wedge j=j$ and the hypothesis that $\kappa_l =1$, $l=0,\ldots,m_1-1,$ and  $\tau_l=1$, $l=0,\ldots,m_2-1$, we get
\begin{align*}
\langle q_j,x^i\rangle&=\sum_{s=1}^{m_1}\gamma_s^ic_{j,s}\sum_{p=s-1}^{[m_1\wedge(m_2+s)]-1}\frac{\binom{m_2}{p-s+1}}{(-2)^{p+1}}
\left(\sum_{l=0}^{j}\beta_{j,l}(-1)^l\rho_{j,l}^s\mathcal{R}_{g_{p+1}}(\theta_{j-l})\right)\\
&\qquad+\sum_{s=m_1+1}^m\gamma_s^id_{j,s}\sum_{p=s-m_1-1}^{[m_2\wedge s]-1}\frac{\binom{m_1}{s-p-1}}{(-2)^{p+1}}\left(\sum_{l=0}^{j}\beta_{j,l}(-1)^l\mathcal{S}_{h_{p+1}}(\theta_{j-l})\right)\\
&=-\sum_{s=1}^{m_1}\gamma_s^ic_{j,s}\sum_{p=s-1}^{[m_1\wedge(m_2+s)]-1}\frac{\binom{m_2}{p-s+1}}{(-2)^{p+1}}\left(
\sum_{l=1}^{m-j}(-1)^{j+l}\beta_{j,j+l}\rho_{j,j+l}^s\mathcal{R}_{g_{p+1}}(\theta_{-l})\right)\\
&\qquad-\sum_{s=m_1+1}^m\gamma_s^id_{j,s}\sum_{p=s-m_1-1}^{[m_2\wedge s]-1}\frac{\binom{m_1}{s-p-1}}{(-2)^{p+1}}\left(\sum_{l=1}^{m-j}(-1)^{j+l}\beta_{j,j+l}\rho_{j,j+l}^s\mathcal{S}_{h_{p+1}}(\theta_{-l})\right)\\
&=-\sum_{l=1}^{m-j}(-1)^{j+l}\beta_{j,j+l}\left(\sum_{s=1}^{m_1}\gamma_s^i\rho_{j,j+l}^s c_{j,s}\sum_{p=s-1}^{[m_1\wedge(m_2+s)]-1}\frac{\binom{m_2}{p-s+1}}{(-2)^{p+1}}\mathcal{R}_{g_{p+1}}(\theta_{-l})\right.\\
&\qquad\qquad+\left.\sum_{s=m_1+1}^m\gamma_s^id_{j,s}\sum_{p=s-m_1-1}^{[m_2\wedge s]-1}\frac{\binom{m_1}{s-p-1}}{(-2)^{p+1}}\mathcal{S}_{h_{p+1}}(\theta_{-l})\right)\\
&=\frac{2^{\alpha+\beta}\Gamma(\beta+1)}{\Gamma(\alpha+\beta+1)}\sum_{l=1}^{m-j}(-1)^{l+1}\beta_{j,j+l}\times\\
&\quad\times \left[\frac{(-1)^l\Gamma(\alpha-l+1)}{\Gamma(\beta-l+1)}\sum_{s=1}^{m_1}\gamma_{s}^i\sum_{p=s-1}^{[m_1\wedge(m_2+s)]-1}\frac{\binom{m_2}{p-s+1}}{(-2)^{p-s+1}}
\mathcal{R}_{g_{p+1}}(\theta_{-l})\right.\\
&\qquad\qquad\left.+\sum_{s=m_1+1}^{m}\gamma_{s}^i\sum_{p=s-m_1-1}^{[m_2\wedge s]-1}\frac{\binom{m_1}{s-p-1}}{(-2)^{p-s+1}}\mathcal{S}_{h_{p+1}}(\theta_{-l})\right].
\end{align*}
In the last step we have used that $\rho_{n,n+j}^sc_{n,s}=\frac{(-1)^j\Gamma(\alpha-j+1)}{\Gamma(\beta-j+1)}d_{n,s}$. Now write
\begin{align}\nonumber
\Xi(i,l)&=\frac{(-1)^l\Gamma(\alpha-l+1)}{\Gamma(\beta-l+1)}\sum_{s=1}^{m_1}\gamma_{s}^i\sum_{p=s-1}^{[m_1\wedge(m_2+s)]-1}\frac{\binom{m_2}{p-s+1}}{(-2)^{p-s+1}}
\mathcal{R}_{g_{p+1}}(\theta_{-l})\\\label{emc0}
&\qquad\qquad+\sum_{s=m_1+1}^{m}\gamma_{s}^i\sum_{p=s-m_1-1}^{[m_2\wedge s]-1}\frac{\binom{m_1}{s-p-1}}{(-2)^{p-s+1}}\mathcal{S}_{h_{p+1}}(\theta_{-l}).
\end{align}
so that
\begin{equation}\label{deta0}
\langle q_j,x^i\rangle=\frac{2^{\alpha+\beta}\Gamma(\beta+1)}{\Gamma(\alpha+\beta+1)}\sum_{l=1}^{m-j}(-1)^{l+1}\beta_{j,j+l}\Xi(i,l).
\end{equation}
Let us write \eqref{emc0} in a more convenient way. Call $\delta_{i,l}$ and $\eta_{i,l}$ the following expressions
\begin{align*}
\delta_{i,l}=&\sum_{s=1}^{m_1}\gamma_{s}^i\sum_{p=s-1}^{[m_1\wedge(m_2+s)]-1}\frac{\binom{m_2}{p-s+1}}{(-2)^{p-s+1}}
\mathcal{R}_{g_{p+1}}(\theta_{-l}),\\
\eta_{i,l}=&\sum_{s=m_1+1}^{m}\gamma_{s}^i\sum_{p=s-m_1-1}^{[m_2\wedge s]-1}\frac{\binom{m_1}{s-p-1}}{(-2)^{p-s+1}}\mathcal{S}_{h_{p+1}}(\theta_{-l}),
\end{align*}
so that $\Xi(i,l)=\frac{(-1)^l\Gamma(\alpha-l+1)}{\Gamma(\beta-l+1)}\delta_{i,l}+\eta_{i,l}$. $\delta_{i,l}$ and $\eta_{i,l}$ can be written in matrix form as
$$
\delta_{i,l}=
\begin{pmatrix}
\gamma_1^{i},\cdots,\gamma_{m_1}^{i}
\end{pmatrix}\begin{pmatrix}
\eta_{11}&\cdots&\eta_{1,m_1}\\
&\ddots&\vdots\\
&&\eta_{m_1,m_1}
\end{pmatrix}\begin{pmatrix}
\mathcal{R}_{g_{1}}(\theta_{-l})\\\vdots\\\mathcal{R}_{g_{m_1}}(\theta_{-l})
\end{pmatrix}
$$
and
$$
\eta_{i,l}=
(-2)^{m_1}\begin{pmatrix}
\gamma_{m_1+1}^{i},\cdots,\gamma_{m}^{i}
\end{pmatrix}\begin{pmatrix}
\eta_{m_1+1,m_1+1}&\cdots&\eta_{m_1+1,m}\\
&\ddots&\vdots\\
&&\eta_{m,m}
\end{pmatrix}\begin{pmatrix}
\mathcal{S}_{h_{1}}(\theta_{-l})\\\vdots\\\mathcal{S}_{h_{m_2}}(\theta_{-l})
\end{pmatrix},
$$
where
$$
\eta_{i,j}=(-2) ^{i-j}\binom{m_2}{j-i},\quad i,j=1,\ldots,m_1,
$$
and
$$
\eta_{m_1+i,m_1+j}=(-2) ^{i-j}\binom{m_1}{j-i},\quad i,j=1,\ldots,m_2.
$$
Using Lemma \ref{nwl} we have that
$$
\begin{pmatrix}
\gamma_1^{i},\cdots,\gamma_{m_1}^{i}
\end{pmatrix}\begin{pmatrix}
\eta_{11}&\cdots&\eta_{1,m_1}\\
&\ddots&\vdots\\
&&\eta_{m_1,m_1}
\end{pmatrix}=\frac{(-1)^i}{2^{m_2}}\begin{pmatrix}
\displaystyle\binom{i}{0},\displaystyle-\binom{i}{1},\cdots,\displaystyle(-1)^{m_1-1}\binom{i}{m_1-1}
\end{pmatrix},
$$
and
\begin{align*}
(-2)^{m_1}&\begin{pmatrix}
\gamma_{m_1+1}^{i},\cdots,\gamma_{m}^{i}
\end{pmatrix}\begin{pmatrix}
\eta_{m_1+1,m_1+1}&\cdots&\eta_{m_1+1,m}\\
&\ddots&\vdots\\
&&\eta_{m,m}
\end{pmatrix}\\
&\qquad\qquad=(-1)^{m_1}\begin{pmatrix}
\displaystyle\binom{i}{0},\displaystyle-\binom{i}{1},\cdots,\displaystyle(-1)^{m_2-1}\binom{i}{m_2-1}
\end{pmatrix}.
\end{align*}
Therefore we can rewrite (\ref{emc0}) in the form
\begin{align}\nonumber
\Xi(i,l)&=\frac{(-1)^{i+l}}{2^{m_2}}\frac{\Gamma(\alpha-l+1)}{\Gamma(\beta-l+1)}\sum_{s=0}^{i\wedge(m_1-1)}(-1)^s\binom{i}{s}\mathcal{R}_{g_{s+1}}(\theta_{-l})\\\label{emc}
&\qquad\qquad+(-1)^{m_1}\sum_{s=0}^{i\wedge(m_2-1)}(-1)^s\binom{i}{s}\mathcal{S}_{h_{s+1}}(\theta_{-l}).
\end{align}
Write now $B$ for the $m\times m$ matrix whose entries are given by
$$
B_{i,j}=\Xi(i,m-j) \quad i,j=0,\ldots, m-1.
$$
Since $\alpha-\max H\neq 0, -1, -2, \ldots,$ and $\beta-\max G\neq 0,-1, -2, \ldots,$ the identity (\ref{betajm}) shows that
$\beta_{i,m}\not =0$. Combining columns, it is easy to see from (\ref{deta0}) that
\begin{equation}\label{deta}
\det A=\frac{(-1)^{\binom{m+1}{2}+m}2^{m(\alpha+\beta)}\Gamma(\beta+1)^m}{\Gamma(\alpha+\beta+1)^m}\left( \prod_{j=0}^{m-1}\beta_{j,m}\right) \det B.
\end{equation}
Using (\ref{emc}), it follows easily that the matrix $B$ is in turn the product of the matrices $C$ and $D$ defined as follows
\begin{align*}
C_{i,s}&=\begin{cases}\displaystyle\frac{(-1)^{i+s}}{2^{m_2}}\binom{i}{s},& s=0,\ldots, m_1-1,\\
\displaystyle(-1)^{m_1+s}\binom{i}{s-m_1},& s=m_1,\ldots, m-1,\end{cases} \quad i=0,\ldots , m-1,\\
D_{s,j}&=\begin{cases}\displaystyle\frac{(-1)^{m-j}\Gamma(\alpha-m+j+1)}{\Gamma(\beta-m+j+1)}R_{g_{s+1}}(\theta_{j-m}),& s=0,\ldots, m_1-1,\\
S_{h_{s+1-m_1}}(\theta_{j-m}),& s=m_1,\ldots, m-1,\end{cases} \quad j=0,\ldots,m-1.
\end{align*}
On the one hand, it is not difficult to prove that $\det C=(-1)^{\binom{m_2}{2}}$. The determinant of the matrix $D$ can be computed, using \eqref{defno}, \eqref{casdet}, \eqref{def1pi} and \eqref{def1qi} for $n=0$ to get
\begin{equation}\label{deta1}
\det D=\frac{\mathfrak p(0)\mathfrak q(0)\Gamma(\alpha-m+1)^{m_1}}{(-1)^{m_1m+\binom{m}{2}}\Gamma(\beta)^{m_1}}\Lambda_{G,H}(0),
\end{equation}
where the polynomials $\mathfrak p$ and $\mathfrak q$ are defined by (\ref{def1pi}) and (\ref{def1qi}), respectively. Combining (\ref{deta}) and (\ref{deta1}), and using (\ref{betajm}), we finally get
\begin{align}\nonumber
\det A&=\frac{\det (C)2^{m(\alpha+\beta)}\Gamma(\beta+1)^m}{\Gamma(\alpha+\beta+1)^m}\left( \prod_{j=0}^{m-1}\beta_{j,m}\right)\frac{\mathfrak p(0)\mathfrak q(0)\Gamma(\alpha-m+1)^{m_1}}{(-1)^{m_1m}\Gamma(\beta)^{m_1}}\Lambda_{G,H}(0)\\\label{spm1}
&=\frac{\mathfrak p(m)\mathfrak q(m)2^{m(\alpha+\beta)}\Gamma(\beta+1)^m\Gamma(\alpha+1)^{m_1}}{(-1)^{\binom{m_2}{2}}\Gamma(\alpha+\beta+1)^m\Gamma(\beta+m)^{m_1}}\left( \prod_{j=0}^{m}\Lambda_{G,H}(j)\right),
\end{align}
from where it follows that $\det A\not=0$, since we are assuming $\Lambda_{G,H}(n)\neq0, n\geq0$ (see \eqref{assum}) and  $\alpha-\max H\neq 0, -1, -2, \ldots,$ and $\beta-\max G\neq 0,-1, -2, \ldots$.
\end{proof}

\bigskip

The case when $\alpha=1,\ldots, \max H$ or/and $\beta=1,\ldots, \max G$
is specially interesting because it includes the Krall-Jacobi polynomials orthogonal with respect to the measures (\ref{kjm1}), (\ref{kjm2}) or (\ref{kjm3}). We only work out here the case when $\alpha=m_2,\ldots, \max H$ and $\beta=m_1,\ldots, \max G$. The other cases can be studied in a similar way.

We next show that we have to change the bilinear forms (\ref{inner2}) and (\ref{inner3}) by transforming a portion of the integral into a discrete Sobolev inner product. First of all, notice that for $g_{l+1}\in G$, $h_{l+1}\in H$, $g_{l+1}\ge \beta$ and $h_{l+1}\ge \alpha$, the rational functions $U_l$ (\ref{Uig}) and $V_l$ (\ref{Vig}) reduce to
\begin{align}\label{ecc1}
U_l(x)&=\phi_l(x)+\frac{\kappa_l}{\Gamma(\beta)}\sum_{s=0}^{\beta -1}
\frac{(\beta-s)_s2^ss!\nu_s^{g_{l+1}}}{(1+x)^{s+1}},\\\nonumber
V_l(x)&=\psi_l(x)+
\frac{\tau_l}{\Gamma(\alpha)}\sum_{s=0}^{\alpha-1}
\frac{(\alpha-s)_s2^ss!\omega_s^{h_{l+1}}}{(1-x)^{s+1}}.
\end{align}
Since $\alpha=m_2,\ldots, \max H$ and $\beta=m_1,\ldots, \max G$, there exist indices $k_1$ and $k_2$, $1\le k_1\le m_1$ and $1\le k_2\le m_2$, such that $g_{l}\ge \beta$ for $l>k_1$ and $h_{l}\ge \alpha$ for $l>k_2$.

Take now $\tilde \alpha, \tilde \beta\not \in \ZZ$ and $p,q\in \PP $, and denote $\langle p,q\rangle _{\tilde 2}$
the bilinear form (\ref{inner2}) defined by $\tilde \alpha$ and $\tilde \beta$.
Using the definition of $U_l$ (\ref{Uig}), we can see that $\langle p,q\rangle _{\tilde 2}$ is formed by certain summands depending on $l$, $l=0,\ldots, m_1-1$, of the form
\begin{equation}\label{ssa}
\int_{-1}^1 p(x)\phi_l(x)d\mu_{\tilde \alpha,\tilde \beta},
\end{equation}
and certain summands depending on $l$ and $s$, $l=0,\ldots, m_1-1$, $s=0,\ldots g_{l+1}$, of the form
\begin{equation}\label{ssb}
\frac{\kappa_l(\tilde \beta-s)_s2^ss!\nu_s^{g_{l+1}}q^{(l)}(-1)}{2^{m_2}l!\Gamma(\tilde \beta)}\int_{-1}^1
\frac{p(x)}{(1+x)^{s+1}}d\mu_{\tilde \alpha,\tilde \beta}.
\end{equation}
Notice that when $\tilde \alpha$ and $\tilde \beta $ tend to $\alpha$ and $\beta$, respectively, each of the summands (\ref{ssa}) has a limit (since $\beta=m_1,\ldots, \max G$). This also happens for the summands (\ref{ssb}) when $l=0,\ldots, k_1-1$ (and hence $g_{l+1}\le \beta-1$), or $l=k_1, \ldots , m_1-1$ and $s=0,\ldots, \beta-1$. Moreover, taking into account (\ref{Uig}) and (\ref{ecc1}), the sum formed by these summands converges to $\langle p,q\rangle _2$.

Consider now one of the summands (\ref{ssb})  corresponding to $l=k_1, \ldots , m_1-1$ (and hence $g_{l+1}\ge\beta$) and $s=\beta, \ldots, g_{l+1}$.
Since
$$
\int_{-1}^1
\frac{p(x)}{(1+x)^{s+1}}d\mu_{\tilde\alpha,\tilde \beta}=\int_{-1}^1
p(x)d\mu_{\tilde \alpha,\tilde \beta-s-1},
$$
we can not take limit when $\tilde \beta$ tends to $\beta$ because $\beta-s-1\le-1$ and we will have divergent integrals. To avoid this problem, we split up $p$ in the form
$$
p(x)=p(x)-\sum_{j=0}^{g_{l+1}-\beta}\frac{p^{(j)}(-1)}{j!}(1+x)^j+\sum_{j=0}^{g_{l+1}-\beta}\frac{p^{(j)}(-1)}{j!}(1+x)^j=r(x)+\sum_{j=0}^{g_{l+1}-\beta}\frac{p^{(j)}(-1)}{j!}(1+x)^j.
$$
On the one hand, we have
$$
\int_{-1}^1r(x)d\mu_{\tilde \alpha,\tilde \beta-s-1}=\int_{-1}^1\frac{r(x)}{(1+x)^{g_{l+1}-\beta+1}}d\mu_{\tilde \alpha, g_{l+1}-s+\tilde\beta-\beta}.
$$
Since $r(x)/(1+x)^{g_{l+1}-\beta+1}$ is still a polynomial and $g_{l+1}-s\ge 0$, the above integral has a finite limit when $\tilde \beta$ tends to $\beta$. That integral is multiplied by the factor $(\tilde \beta-s)_s$ (see (\ref{ssb})), which goes to zero as $\tilde \beta$ tends to $\beta$ (since $s\ge \beta$).

On the other hand, for $j=0,\ldots, g-\beta,$ we have
$$
\int_{-1}^1(1+x)^jd\mu_{\tilde \alpha,\tilde \beta-s-1}=\int_{-1}^1d\mu_{\tilde \alpha,\tilde \beta+j-s-1}=2^{\tilde \alpha+\tilde \beta+j-s}
\frac{\Gamma(\tilde \alpha+1)\Gamma(\tilde \beta-s+j)}{\Gamma(\tilde \alpha+\tilde \beta-s+j+1)}.
$$
Multiplying $\Gamma(\tilde \beta-s+j)$ by $(\tilde \beta-s)_s$, we get
$$
(\tilde \beta-s)_s\Gamma(\tilde \beta-s+j)=(\tilde \beta-s)_j\Gamma (\tilde \beta),
$$
which tends to $(\beta-s)_j\Gamma (\beta)$ as $\tilde \beta$ tends to $\beta$.

Hence, we have proved that
$$
\lim_{\tilde\alpha\to\alpha,\tilde\beta\to\beta}\int_{-1}^1
\frac{(\tilde \beta-s)_s p(x)}{(1+x)^{s+1}}d\mu_{\tilde\alpha,\tilde \beta}=\Gamma(\alpha+1)\Gamma (\beta)\sum_{j=0}^{g-\beta}\frac{2^{\alpha+\beta+j-s}p^{(j)}(-1)}{j!}
\frac{(\beta-s)_j}{\Gamma(\alpha+\beta-s+j+1)}.
$$
This gives the discrete Sobolev bilinear form
\begin{equation*}\label{inner2S}
\langle p,q\rangle_{2S}=2^{\alpha+\beta}\Gamma(\alpha+1)
\sum_{l=0}^{m_1-1}\kappa_l\frac{q^{(l)}(-1)}{2^{m_2}l!}\sum_{j=0}^{g_{l+1}-\beta}
\frac{2^jp^{(j)}(-1)}{j!}\sum_{s=\beta +j}^{g_{l+1}}
\frac{(\beta-s)_js!}{\Gamma(\alpha+\beta-s+j+1)}\nu_s^{g_{l+1}}.
\end{equation*}
Proceeding in the same way with the bilinear form (\ref{inner3}), we get the discrete Sobolev bilinear form
\begin{equation*}\label{inner3S}
\langle p,q\rangle_{3S}=2^{\alpha+\beta}\Gamma(\beta+1)
\sum_{l=0}^{m_2-1}\tau_l\frac{q^{(l)}(1)}{(-1)^{m_1+l}l!}\sum_{j=0}^{h_{l+1}-\alpha}
\frac{2^jp^{(j)}(1)}{(-1)^jj!}\sum_{s=\alpha +j}^{h_{l+1}}
\frac{(\alpha-s)_js!}{\Gamma(\alpha+\beta-s+j+1)}\omega_s^{h_{l+1}}.
\end{equation*}
We then define the bilinear form
\begin{equation}\label{innerzx}
\langle p,q\rangle_S=\langle p,q\rangle_1+\langle p,q\rangle_2+\langle p,q\rangle_{2S}+\langle p,q\rangle_3+\langle p,q\rangle_{3S}.
\end{equation}

It turns out that Lemma \ref{lemwenok} is still true for this inner product.

\begin{lemma}\label{lmn1}
Let $k,n\geq0$, $j=0,1,\ldots,n$ and $n-j\ge k$. For $i=1,\ldots,m_1,$ we have
$$
\langle (1+x)^kJ_{n-j}^{\alpha,\beta},b_i\rangle_S=\frac{c_{n,i}\rho_{n,j}^i}{(-1)^j}\left[\sum_{l=i-1}^{[m_1\wedge (m_2+i)]-1}\frac{\kappa_l\binom{m_2}{l-i+1}}{(-2)^{l+1}} \sum_{s=k}^{g_{l+1}}2^k\nu_s^{g_{l+1}}(\beta-s)_k(s-k+1)_ku_{s-k}^{\alpha}(n-j)\right],
$$
where $c_{n,i}=(-1)^{m+i}2^{\alpha+\beta+i}
\Gamma(\beta+1)\Gamma(n+\alpha-m+1)/(\Gamma(\alpha+\beta+1)\Gamma(n+\beta))$ and $\rho_{x,j}^i$ is defined by \eqref{defno}. Similarly, for $i=m_1+1,\ldots,m,$ we have
$$
\langle (1-x)^kJ_{n-j}^{\alpha,\beta},b_i\rangle_S=\frac{d_{n,i}}{(-1)^j}\left[\sum_{l=i-m_1-1}^{[m_2\wedge i]-1}\frac{\tau_l\binom{m_1}{i-l-1}}{(-2)^{l+1}}\sum_{s=k}^{h_{l+1}}2^k\omega_s^{h_{l+1}}(\alpha-s)_k(s-k+1)_k
u_{s-k}^{\alpha}(n-j)\right],
$$
where $d_{n,i}=(-1)^{n+i}2^{\alpha+\beta+i}\Gamma(\beta+1)/\Gamma(\alpha+\beta+1)$.
\end{lemma}

\begin{proof}
Call $v=n-j$. For $i=1,\ldots,m_1,$ we have that $\langle (1+x)^kJ_{v}^{\alpha,\beta},b_i\rangle_3=\langle (1+x)^kJ_{v}^{\alpha,\beta},b_i\rangle_{3S}=0$. Therefore
$$
\langle (1+x)^kJ_{v}^{\alpha,\beta},b_i\rangle_S=\langle (1+x)^kJ_{v}^{\alpha,\beta},b_i\rangle+\langle (1+x)^kJ_{v}^{\alpha,\beta},b_i\rangle_{2S},
$$
where the inner product $\langle\cdot,\cdot\rangle$ is defined by \eqref{inner}. Following the same steps as in the proof of Lemma \ref{lemwenok} we have that
\begin{align*}
\langle& (1+x)^kJ_{v}^{\alpha,\beta},b_i\rangle=\frac{c_{v+j,i}\rho_{v+j,j}^i}{(-1)^j}\left[\sum_{l=i-1}^{[m_1\wedge (m_2+i)]-1}\frac{\kappa_l\binom{m_2}{l-i+1}}{(-2)^{l+1}} \sum_{s=k}^{\beta-1}2^k\nu_s^{g_{l+1}}(\beta-s)_k(s-k+1)_ku_{s-k}^{\alpha}(v)\right]\\
&=\frac{2^{\alpha+\beta}\Gamma(\beta+1)\Gamma(v+\alpha+1)}{\Gamma(\alpha+\beta+1)\Gamma(v+\beta+1)}\left[\sum_{l=i-1}^{[m_1\wedge (m_2+i)]-1}\frac{\kappa_l\binom{m_2}{l-i+1}}{(-2)^{l-i+1}} \sum_{s=k}^{\beta-1}2^k\nu_s^{g_{l+1}}(\beta-s)_k(s-k+1)_ku_{s-k}^{\alpha}(v)\right].
\end{align*}
On the other hand we have, using \eqref{propb1}, \eqref{fff1}, \eqref{basu}, \eqref{gammas} and \eqref{formulaca} that
\begin{align*}
\langle& (1+x)^kJ_{n-j}^{\alpha,\beta},b_i\rangle_{2S}=2^{\alpha+\beta}\Gamma(\alpha+1)\sum_{l=i-1}^{[m_1\wedge (m_2+i)]-1}\frac{\kappa_l\binom{m_2}{l-i+1}}{(-2)^{l-i+1}}\sum_{j=0}^{g_{l+1}-\beta}\frac{(-1)^{j+k}j!}{2^{-k}\binom{\alpha+\beta}{\beta}}\times\\
&\times\binom{v+\alpha+\beta}{\alpha}\binom{v+\beta}{v-j+k}\binom{v+\alpha+\beta+j-k}{j-k}\sum_{s=\beta+j}^{g_{l+1}}\frac{(\beta-s)_js!}{\Gamma(\alpha+\beta-s+j+1)}\nu_s^{g_{l+1}}\\
&=\frac{2^{\alpha+\beta}\Gamma(\beta+1)\Gamma(v+\alpha+1)}{\Gamma(\alpha+\beta+1)\Gamma(v+\beta+1)}\times\\
&\times\left[\sum_{l=i-1}^{[m_1\wedge (m_2+i)]-1}\frac{\kappa_l\binom{m_2}{l-i+1}}{(-2)^{l-i+1}}\sum_{j=0}^{g_{l+1}-\beta}\frac{(-1)^{j+k}2^k\Gamma(\alpha+1)\Gamma(v+\beta+1)\Gamma(v+\alpha+\beta+j-k+1)}{\Gamma(u+\alpha+1)\Gamma(u-j+k+1)\Gamma(\beta+j-k+1)\Gamma(j-k+1)}\times\right.\\
&\qquad\times\left.\sum_{s=\beta+j}^{g_{l+1}}\frac{(\beta-s)_js!}{\Gamma(\alpha+\beta-s+j+1)}\nu_s^{g_{l+1}}\right]\\
&=\frac{2^{\alpha+\beta}\Gamma(\beta+1)\Gamma(v+\alpha+1)}{\Gamma(\alpha+\beta+1)\Gamma(v+\beta+1)}\times\\
&\times\left[\sum_{l=i-1}^{[m_1\wedge (m_2+i)]-1}\frac{\kappa_l\binom{m_2}{l-i+1}}{(-2)^{l-i+1}}\sum_{s=\beta}^{g_{l+1}}2^k\nu_s^{g_{l+1}}\times\right.\\
&\times\left.\sum_{j=0}^{s-\beta}\frac{(-1)^{j+k}\Gamma(\alpha+1)\Gamma(v+\beta+1)\Gamma(v+\alpha+\beta+j-k+1)(\beta-s)_js!}{\Gamma(u+\alpha+1)\Gamma(u-j+k+1)\Gamma(\beta+j-k+1)\Gamma(j-k+1)\Gamma(\alpha+\beta-s+j+1)}\right]\\
&=\frac{2^{\alpha+\beta}\Gamma(\beta+1)\Gamma(v+\alpha+1)}{\Gamma(\alpha+\beta+1)\Gamma(v+\beta+1)}\times\\
&\times\left[\sum_{l=i-1}^{[m_1\wedge (m_2+i)]-1}\frac{\kappa_l\binom{m_2}{l-i+1}}{(-2)^{l-i+1}}\sum_{s=\beta}^{g_{l+1}}2^k\nu_s^{g_{l+1}}\frac{(\beta-s)_ k(s-k+1)_k\Gamma(v+s-k+1)\Gamma(v+\alpha+\beta+1)}{\Gamma(v+1)\Gamma(v+\alpha+\beta-s+k+1)}\right]\\
&=\frac{2^{\alpha+\beta}\Gamma(\beta+1)\Gamma(v+\alpha+1)}{\Gamma(\alpha+\beta+1)\Gamma(v+\beta+1)}\left[\sum_{l=i-1}^{[m_1\wedge (m_2+i)]-1}\frac{\kappa_l\binom{m_2}{l-i+1}}{(-2)^{l-i+1}} \sum_{s=\beta}^{g_{l+1}}2^k\nu_s^{g_{l+1}}(\beta-s)_k(s-k+1)_ku_{s-k}^{\alpha}(v)\right].
\end{align*}
In the third step we have interchanged the order of summation and in the fourth step we have used \eqref{formulaca} written in terms of Gamma functions. Therefore adding the previous expression to $\langle (1+x)^kJ_{v}^{\alpha,\beta},b_i\rangle$ we get $\langle (1+x)^kJ_{v}^{\alpha,\beta},b_i\rangle_S$. The second identity can be proved in a similar way.

\end{proof}

\begin{theorem}\label{tho2j}
Let $\alpha$ and $\beta$ be positive integers satisfying $\alpha=m_2,\ldots, \max H$ and $\beta=m_1,\ldots, \max G$
and assume that conditions \eqref{assum} hold.
Then, for $n\geq m$, the polynomials $(q_n)_n$ defined by \eqref{iquss} satisfy the following orthogonality properties with respect to the bilinear from \eqref{innerzx}:
\begin{align*}
\langle q_n,q_i\rangle_S&=0,\quad i=0,1,\ldots,n-1,\\
\langle q_n,q_n\rangle_S&\neq0.
\end{align*}
Moreover,  if we assume that $\kappa_l =1$, $l=0,\ldots,m_1-1$ and  $\tau_l=1$, $l=0,\ldots,m_2-1$, then the matrix $A$ with entries $A_{i,j}=\langle q_j,x^i\rangle_S$, $i,j=0,\ldots, m-1$, is nonsingular.
\end{theorem}

\begin{proof}
The proof of the first part is similar to that of Theorem \ref{thoj}, but using now Lemma \ref{lmn1} instead of Lemma \ref{lemwenok}. The second part can be proved by passing to the limit in Lemma \ref{cdsm}. Indeed, take $\tilde \alpha, \tilde \beta\not \in \ZZ$ and write $\tilde A$ for the corresponding matrix associated to $\tilde \alpha, \tilde \beta$. We have that $\tilde A$ tends to $A$ as $\tilde \alpha, \tilde \beta$ tends to $\alpha, \beta$, respectively. According to (\ref{spm1}) we have
$$
\det A=\frac{\mathfrak p(m)\mathfrak q(m)2^{m(\alpha+\beta)}\Gamma(\beta+1)^m\Gamma(\alpha+1)^{m_1}}{(-1)^{\binom{m_2}{2}}\Gamma(\alpha+\beta+1)^m\Gamma(\beta+m)^{m_1}}\left( \prod_{j=0}^{m}\Lambda_{G,H}(j)\right),
$$
where the polynomials $\mathfrak p$ and $\mathfrak q$ are defined by (\ref{def1pi}) and (\ref{def1qi}), respectively.
From here it follows that $\det A\not=0$, since we are assuming $\Lambda_{G,H}(n)\neq0, n\geq0$ (see \eqref{assum}) and for $\alpha=m_2,\ldots, \max H$ and $\beta=m_1,\ldots, \max G$, we have $\mathfrak p(m)\mathfrak q(m)\not =0$.
\end{proof}

\section{Recurrence relations}\label{jrr}
In the next theorem we prove that the Jacobi-type polynomials (\ref{iquss}) are bispectral. We first consider the case $\alpha-\max H\neq 0, -1, -2, \ldots,$ and $\beta-\max G\neq 0,-1, -2, \ldots$.

\begin{theorem}\label{t6.9}
Let $\alpha$ and $\beta$ be real numbers with $\alpha-\max H\neq 0, -1, -2, \ldots,$ and $\beta-\max G\neq 0,-1, -2, \ldots$. Let $Q\in\mathbb{R}[x]$ be a polynomial of degree $s$ satisfying that
$(1+x)^{\max G }(1-x)^{\max H }$ divides $Q'$. Then the sequence $(q_n)_n$ in \eqref{iquss} satisfies the recurrence relation
\begin{equation}\label{frov}
Q(x)q_n(x)=\sum_{j=-s}^s\gamma_{n,j}q_{n+j}(x),\quad \gamma_{n,s},\gamma_{n,-s}\neq0.
\end{equation}
\end{theorem}

\begin{proof}

Let us take $\kappa_l =1$, $l=0,\ldots,m_1-1,$ and  $\tau_l=1$, $l=0,\ldots,m_2-1$ (see (\ref{Uig}) and (\ref{Vig})). We proceed in five steps.

\smallskip

\noindent
\textsl{First step.} For $n\ge m+s$ and $i=1,\ldots , n-s$, we have that $\langle Qq_n,b_i\rangle =0$ and $\langle Qq_n,b_{n-s+1}\rangle \not =0$, where $(b_i)_i$ is the basis defined by (\ref{basisb}). Indeed, from the hypothesis we have that $(Q(x)-Q(1))(1-x)^{-j-1}$, $j=0,\ldots , \max H$, and $(Q(x)-Q(-1))(1+x)^{-j-1}$, $j=0,\ldots , \max G$, are always polynomials. For $i=1,\ldots,m_1,$ we have
\begin{align*}
\langle Qq_n,b_i\rangle&=\langle (Q(x)-Q(-1))q_n,b_i\rangle+Q(-1)\langle q_n,b_i\rangle\\
&=\langle (1+x)^{\max G+1}r(x)q_n,b_i\rangle+Q(-1)\langle q_n,b_i\rangle,
\end{align*}
for some polynomial $r$ of degree $s-\max G-1$ with $r(-1)\neq0$. The result then holds following the same steps as in the proof of Theorem \ref{thoj}. The rest of cases are similar.

\smallskip

\noindent
\textsl{Second step.} For $n\ge m+s$ and $i=0,\ldots , n-s-1$, we have that $\langle Qq_n,q_i\rangle =0$ and $\langle Qq_n,q_{n-s}\rangle \not =0$. Indeed, it follows from the first step, taking into account that $q_i=\sum_{j=1}^{m\vee(1+i)}\xi_{i,j}b_j$, and hence, for $n\ge m+s$ and $i=0,\ldots , n-s-1,$ we have that $1\le j\le n-s$ for $j=1,\ldots m\vee(1+i)$.

\smallskip

\noindent
\textsl{Third step.} The recurrence formula (\ref{frov}) holds for $n\ge m+s$.
Since $\deg q_n=n$, we can always write
\begin{equation}\label{ole}
Q(x)q_n(x)=\sum_{j=-n}^s\gamma_{n,j}q_{n+j},
\end{equation}
with $\gamma_{n,s}\not =0$, $n\ge 0$. Take now $n\ge m+s$ and $i\le m-1$. Using the orthogonality conditions of Theorem \ref{thoj}, we get from (\ref{ole})
\begin{align*}
\langle Q(x)q_n(x),x^i\rangle &=\sum_{j=-n}^s\gamma_{n,j}\langle q_{n+j}(x),x^i\rangle=
\sum_{j=-n}^{m-n-1}\gamma_{n,j}\langle q_{n+j}(x),x^i\rangle\\
&=\sum_{j=0}^{m-1}\gamma_{n,j-n}\langle q_{j}(x),x^i\rangle.
\end{align*}
The second step then gives  $\langle Q(x)q_n(x),x^i\rangle =0$, and therefore
$$
0=\sum_{j=0}^{m-1}\gamma_{n,j-n}\langle q_{j}(x),x^i\rangle,\quad i=0,\ldots, m-1.
$$
For each $n\ge m+s$, this can be seen as a linear system of $m$ homogeneous equations in the $m$ unknowns $\gamma_{n,j-n}$, $j=0,\ldots , m-1$. Lemma \ref{cdsm} gives  that the coefficient matrix $A=(\langle q_{j}(x),x^i\rangle)_{i,j=0}^{m-1}$ is nonsingular, and hence we deduce $\gamma_{n,j-n}=0$, for $j=0,\ldots , m-1,$ and $n\ge m+s$. Using this, the recurrence relation (\ref{ole}) reduces to
\begin{equation*}\label{olex}
Q(x)q_n(x)=\sum_{j=-n+m}^s\gamma_{n,j}q_{n+j}.
\end{equation*}
Using again the second step and the orthogonality conditions of Theorem \ref{thoj} for $i=m$, we get
$$
0=\langle Qq_n,q_m\rangle=\gamma_{n,m-n}\langle q_m,x^m\rangle.
$$
Since $\langle q_m,x^m\rangle\not =0$ (see Theorem \ref{thoj}), we deduce that $\gamma_{n,m-n}=0$. In the same way, we can prove that $\gamma_{n,j}=0$ for $j=-n+m+1,\ldots, -s-1$, and $\gamma_{n,-s}\not =0$. Hence, the polynomials $q_n$, $n\ge m+s$, satisfy the recurrence relation (\ref{frov})
for $n\ge m+s$.

\smallskip

\noindent
\textsl{Fourth step.} Let $P$ be a polynomial of degree $s$ and write
\begin{equation}\label{Pqnrep}
P(x)q_n(x)=\sum_{j=-n}^s\gamma_{n,j}q_{n+j}(x).
\end{equation}
Then for fixed $j\le s$, the recurrence coefficients $\gamma _{n,j}$ are rational functions of $n$ for $n\ge \max\{0,-j\}$. Indeed, since the sequences $\beta_{n,j}$ (\ref{betaj}) are rational functions of $n$ and the recurrence coefficients of the Jacobi polynomials are also rational functions of $n$, the result can be proved in a similar way as in \cite[Lemma 2.5]{ducb}.

\smallskip

\noindent
\textsl{Fifth step.} The recurrence formula (\ref{frov}) holds for $n\ge 0$. Since $Q$ has degree $s$, we always have a representation of $Q(x)q_n(x)$ like in \eqref{Pqnrep}. For a fixed $j\le -s-1$, from the third step, we deduce that $\gamma_{n,j}=0$ for $n\ge m+s$. Since $\gamma_{n,j}$ is a rational function of $n$ for $n\ge s+1$, it follows that $\gamma_{n,j}=0$ for $n\ge s+1$ as well. Then the recurrence formula (\ref{frov}) holds for $n\ge s+1$. For $n=0,\ldots, s,$ the recurrence formula (\ref{frov}) always holds.

\end{proof}

The orthogonality properties proved in Section \ref{jorth} impose some constrains on the recurrence relations satisfied by the polynomials (\ref{iquss}) (the proofs of the following results are similar to that of Theorem 4.2 and Corollaries 4.3 and 4.4 in \cite{ddIlb}, hence we only include here the proof of the
Theorem \ref{jpzt}).

\begin{theorem}\label{jpzt} Let $\alpha,\beta $ be real numbers with $\alpha-\max H\neq 0, -1, -2, \ldots,$ and $\beta-\max G\neq 0,-1, -2, \ldots$, and assume that conditions (\ref{assum}) hold. Let $Q$ be the polynomial $Q(x)=\sum_{k=u}^v\sigma_k(1+x)^k$, with $u\le v$ and $\sigma_u,\sigma_v\not =0$. If there exists $\hat g\in G$ such $\hat g-u\not \in G$ and $\hat g-u\ge 0$ then the polynomials $(q_n)_n$ in \eqref{iquss} do not satisfy a recurrence relation of the form (\ref{qrr}). Analogously,
let $Q$ be the polynomial $Q(x)=\sum_{k=w}^v\tilde \sigma_k(1-x)^k$, with $w\le v$ and $\tilde \sigma_w,\tilde \sigma_v\not =0$. If there exists $\hat h\in H$ such $\hat h-w\not \in H$ and $\hat h-w\ge 0$ then the polynomials $(q_n)_n$ do not satisfy a recurrence relation of the form (\ref{qrr}).
\end{theorem}

\begin{proof}
We proceed by \textit{reductio ad absurdum}.
Assume that the sequence $(q_n)_n$ satisfies the recurrence relation (\ref{qrr}). Write $0\le \hat l\le m_1-1$ for the index such that $\hat g=g_{\hat l+1}$. Using Theorem \ref{thoj} we get $0=\langle Q(x)q_n,b_{\hat l+1}\rangle$, for $n$ big enough. Using now Lemma \ref{lemwenok} we can write
\begin{align*}
0&=\langle Q(x)q_n,b_{\hat l+1}\rangle=c_{n,\hat l+1}\sum_{j=0}^m\beta _{n,j}(-1)^j\rho_{n,j}^{\hat l+1}\Upsilon(n-j),
\end{align*}
where $\Upsilon$ is the polynomial
$$
\Upsilon(x)=\sum_{l=\hat l}^{[m_1\wedge(m_2+\hat l+1)]-1}\frac{\kappa_l\binom{m_2}{l-\hat l}}{(-2)^{l+1}}\sum_{k=u}^v\sigma_k \sum_{s=k}^{g_{\hat l+1}}2^k\nu_s^{g_{\hat l+1}}(\beta-s)_k(s-k+1)_ku_{s-k}^{\alpha}(x).
$$
Taking $\kappa_{\hat l}=1$ and $\kappa_l=0$ for $l\neq\hat l$, we get that $\deg\Upsilon=2(\hat g-u)$. Since each $u_i^{\alpha}(x)$ is actually a polynomial in $\theta _x$, there exists a polynomial $P$ such that $\Upsilon(x)=P(\theta_x)$.
On the other hand, the degree of $u_i^\alpha (x)$ is $2i$, and since $\sigma_u 2^u\nu_{\hat g}^{\hat g}(\beta-\hat g)_u(\hat g-u+1)_u\not =0$, we conclude that
$P$ is a polynomial of degree $\hat g-u$ satisfying
\begin{align*}
0&=\sum_{j=0}^m\beta _{n,j}(-1)^j\rho_{n,j}^{\hat l+1}P(\theta_{n-j}).
\end{align*}
If we set $Y_0(x)=P(x)$, $Y_l(x)=R_{g_l}(x)$, $l=1,\ldots , m_1,$ and $Y_{m_1+l}(x)=S_{h_l}(x)$, $l=1,\ldots , m_2$, we get from (\ref{iqussy}) that
\begin{align*}
W_a^Y(x)&=\sum_{j=0}^m\beta _{x,j}(-1)^j\rho_{x,j}^{\hat l+1}P(\theta_{x-j})=0,
\end{align*}
which it is a contradiction because since $\deg P=\hat g-u\not\in G$, then $W_a^Y$ is a polynomial of
degree
$$
d=2\left[ \hat g-u+\sum_{g\in G}g+\sum_{h\in H}h-\binom{m_1+1}{2}-\binom{m_2}{2}\right]>0.
$$
The second part is similar but now starting with an index $0\le \hat l\le m_2-1$ such that $\hat h=h_{\hat l+1}$, using that $0=\langle Q(x)q_n,b_{m_1+\hat l+1}\rangle$ for $n$ big enough, using the second part of Lemma \ref{lemwenok}, taking $\tau_{\hat l}=1$ and $\tau_l=0$ for $l\neq\hat l$ and finishing with \eqref{iqussy2}.

\end{proof}

\begin{corollary}\label{jpjc} Let $\alpha,\beta $ be real numbers satisfying $\alpha-\max H\neq 0, -1, -2, \ldots,$ and $\beta-\max G\neq 0,-1, -2, \ldots$, and assume
that conditions (\ref{assum}) hold. Then the polynomials $(q_n)_n$ never satisfy a three-term recurrence relation and hence they are not orthogonal with respect to any measure.
\end{corollary}

Using Theorem \ref{jpzt} we can also characterize the algebra of operators $\al_n$ when $G$ and $H$ are segments.

\begin{corollary}\label{coroseg}
Let $\alpha,\beta $ be real numbers satisfying $\alpha-\max H\neq 0, -1, -2, \ldots,$ and $\beta-\max G\neq 0,-1, -2, \ldots$, and assume
that conditions (\ref{assum}) hold. If $G$ and $H$ are segments then
$$
\alp_n=\{Q\in\mathbb{R}[x]: \mbox{$(1+x)^{\max G}(1-x)^{\max H}$ divides $Q'$}\}.
$$
\end{corollary}
If $G$ or $H$ are not segments the algebra $\alp_n$ can have a more complicated structure, as the following example shows.
\begin{example}\label{ex1}
Take $G=\{1,3\}$, $H=\{1\}$, $\alpha=1/2,\beta=1/3$, and
$$
R_1(x)=x+1,\quad R_3(x)=x^3+x^2/3+2x/3+1, \quad \quad S_1(x)=x+1/2.
$$
Using Maple one can see that the polynomials $(q_n)_n$ satisfy recurrence relations of the form (\ref{qrr}) for $Q_0(x)=\frac{x}{135}(135x^3+244x^2-270x-732)$,
and $Q_0'$ is not divisible by $(1+x)^3(1-x)$.
Computational evidence suggests that
$$
\alp_n=\{Q(x)+c_0Q_0(x) :\hbox{$(1+x)^3(1-x)$ divides $Q'$ and $c_0\in\mathbb{R}$}\}.
$$
\end{example}

\medskip 

The case when $\alpha=1,\ldots, \max H,$ or/and $\beta=1,\ldots, \max G,$ is specially interesting because it includes the Krall-Jacobi polynomials orthogonal with respect to the measures (\ref{kjm1}), (\ref{kjm2}) or (\ref{kjm3}). As in the previous section, we only work out here the case $\alpha=m_2,\ldots, \max H,$ and $\beta=m_1,\ldots, \max G$ (the other cases can be studied in a similar way). The orthogonality conditions in Theorem \ref{tho2j} lead us to an improvement of Theorem \ref{t6.9} (the proof is similar to that of Theorem \ref{t6.9} and it is omitted).

\begin{corollary} Let $\alpha,\beta$ be positive integers satisfying $\alpha=m_2,\ldots, \max H,$ and $\beta=m_1,\ldots,\max G,$ and assume that conditions \eqref{assum} hold. Write $\rho_1=\max\{\max G-\beta+1,\beta\}$ and $\rho_2=\max\{\max H-\alpha+1,\alpha\}$.
Let $Q\in\mathbb{R}[x]$ be a polynomial of degree $s$ satisfying that
$(1+x)^{\rho_1}(1-x)^{\rho_2 }$ divides $Q'$. Then the sequence $(q_n)_n$ in \eqref{iquss} satisfies the recurrence relation of the form (\ref{qrr}).
\end{corollary}

Theorem \ref{tho2j} also allows us to prove that the only Jacobi type polynomials which are orthogonal with respect a measure on the real line are the Krall-Jacobi polynomials (the proof is similar to that of Theorem  5.5 in \cite{ddIlb}).

\begin{theorem} Let $\alpha,\beta $ be two positive integers satisfying $\alpha=m_2,\ldots, \max H,$ and $\beta=m_1,\ldots, \max G$, and assume
that conditions \eqref{assum} hold. Then the sequence $(q_n)_n$ only satisfies a three-term recurrence relation as in (\ref{ttrr}) when the polynomials $\R_g(x)$ and $\Ss_h(x)$ have the form (\ref{pel1}) and (\ref{pel2}), respectively.
\end{theorem}

\begin{proof}
We proceed by \textit{reductio ad absurdum}.
Assume that the sequence $(q_n)_n$ satisfies the three-term recurrence relation
of the form (\ref{ttrr}).
Theorem \ref{tho2j} gives $0=\langle (1+x)q_n,b_1\rangle$, for $n\ge m+1$. Using Lemma \ref{lmn1} (for $k=i=1$) we have
\begin{align*}
0=&\langle (1+x)q_n,b_1\rangle\\
&=c_{n,1}\sum_{j=0}^m(-1)^j\beta_{n,j}\rho_{n,j}^1\left[\sum_{l=0}^{[m_1\wedge (m_2+1)]-1}\frac{\kappa_l\binom{m_2}{l}}{(-2)^{l+1}}\sum_{s=1}^{g_{l+1}}2\nu_s^{g_{l+1}}(\beta-s)su_{s-1}^{\alpha}(n-j)\right].
\end{align*}
For $\hat g\in G$, write $0\le \hat l\le m_1-1,$ for the index such that $\hat g=g_{\hat l+1}$. Taking $\kappa_{\hat l}=1$ and $\kappa_l=0$ for $l\neq\hat l$, we get
\begin{equation}\label{aet}
0=\sum_{j=0}^m(-1)^j\beta_{n,j}\rho_{n,j}^1r_{\hat g}(n-j),
\end{equation}
where $r_{\hat g}$ is the polynomial
\begin{equation*}\label{aet2}
r_{\hat g}(x)=2\sum_{s=1}^{\hat g}\nu_s^{\hat g}(\beta-s)su_{s-1}^{\alpha}(x).
\end{equation*}
Since $u_i^\alpha(x)$ is a polynomial in $\theta_x$, there exists a polynomial $R_{\hat g}$ such that $R_{\hat g}(\theta_x)=r_{\hat g}(x)$. If we set $Y_0(x)=R_{\hat g}(x)$, $Y_{l}(x)=\R_{g_l}(x)$, $l=1,\ldots , m_1$ and $Y_{m_1+l}(x)=\Ss_{h_l}(x)$, $l=1,\ldots ,m_2$, we get from (\ref{iqussy}) and \eqref{aet} that
\begin{align*}
W_a^Y(x)&=\sum_{j=0}^m\beta _{x,j}(-1)^j\rho_{x,j}^iR_{\hat g}(\theta_{x-j})=0.
\end{align*}
Consider now $\hat g=g_1=\min G$. If $\beta\not =g_1$, the polynomial $R_{\hat g}$ has degree $g_1-1\not \in G$  and so $0=W_a^Y(x)$ is a contradiction because $W_a^Y$ is a polynomial of degree (see (\ref{gcd}))
$
2\left(\hat g-1+\sum_{g\in G}g+\sum_{h\in H}h-\binom{m_1+1}{2}-\binom{m_2}{2}\right)\ge 0.
$
Hence $\beta=g_1$. In a similar way we can deduce that $\nu_l^{\hat g}=0$, $l=1,\ldots g_1-1,$ and therefore $r_{\hat g}=0$. This gives for $\R_{g_1}$ the form
\begin{equation}\label{aet3}
\R_{g_1}(\theta_x)=u_\beta^\alpha(x)+a_0,
\end{equation}
for certain real number $a_0$.

Take now $\hat g=g_2$, the second element of $G$. Since $\beta\not =g_2$, the polynomial $R_{\hat g}$ has degree $g_2-1$ and so $g_2-1\not \in G$, otherwise $0=W_a^Y(x)$ is a contradiction because $W_a^Y$ would be a polynomial of degree $2\left(g_2-1+\sum_{g\in G}g+\sum_{h\in H}h-\binom{m_1+1}{2}-\binom{m_2}{2}\right)> 0$ (see (\ref{gcd})). Hence, we conclude that $g_2-1 \in G$ and so $g_2=\beta+1$. Proceeding as before, we can then conclude that
\begin{equation*}\label{aet4}
\R_{g_2}(\theta_x)=u_{\beta+1}^\alpha(x)-\frac{a_0(\beta+1)}{\beta-1}u_1^\alpha(x)+a_1,
\end{equation*}
for certain real number $a_1$.
We can now proceed in the same way to get \eqref{pel1}.

Similarly, for the polynomials $\Ss(\theta_x)$ in \eqref{pel2}, we start with $0=\langle (1-x)q_n,b_{m_1+1}\rangle$, for $n$ big enough and for $\hat h\in H$ we write $0\le \hat l\le m_2-1,$ for the index such that $\hat h=h_{\hat l+1}$. Taking $\tau_{\hat l}=1$ and $\tau_l=0$ for $l\neq\hat l$ we proceed then in the same way as before but using $\alpha$ instead of $\beta$ and \eqref{iqussy2} instead of \eqref{iqussy}.
\end{proof}

If $\alpha=1,\ldots, \max H,$ or/and $\beta=1,\ldots, \max G$, Theorem \ref{jpzt} is no longer true. Even Corollary \ref{coroseg} also fails in this case, as the following example shows.
\begin{example}\label{ex2}
Take $G=\{1\}$, $H=\{2\}$, $\alpha=2,\beta=1$, and
$$
R_1(x)=x+1,\quad \quad S_2(x)=x^2+2x/3+1/2.
$$
Using Maple one can see that the polynomials $(q_n)_n$ satisfy recurrence relations of the form (\ref{qrr}) for 
\begin{align*}
Q_0(x)&=x(x-2),\quad Q_1(x)=\frac{x^2}{2}(2x-3).
\end{align*}
However $Q_0'(x)=-2(1-x)$ and $Q_1'(x)=-3x(1-x)$ are not divisible by $(1+x)(1-x)^2$.
Computational evidence suggests that
$$
\alp_n=\{Q(x)+c_0Q_0(x)+c_1Q_1(x) :\hbox{$(1+x)(1-x)^2$ divides $Q'$ and $c_0,c_1\in\mathbb{R}$}\}.
$$
\end{example}

\end{document}